\documentclass[11pt]{article}
\usepackage{amsmath, amsthm, amsfonts, amssymb, amscd}
\usepackage{graphicx}
\usepackage{hyperref}

\theoremstyle{plain}

\newtheorem{thm}{Theorem}[section]
\newtheorem{prop}[thm]{Proposition}
\newtheorem{cor}[thm]{Corollary}
\newtheorem{lem}[thm]{Lemma}

\newtheorem{defn}{Definition}

\newenvironment{nonumthm}[2][Theorem]{\begin{trivlist}
\item[\hskip \labelsep {\bfseries #1}\hskip \labelsep {\bfseries #2}]}{\end{trivlist}}

\newenvironment{nonumcor}[2][Corollary]{\begin{trivlist}
\item[\hskip \labelsep {\bfseries #1}\hskip \labelsep {\bfseries #2}]}{\end{trivlist}}

\newenvironment{nonumprop}[2][Proposition]{\begin{trivlist}
\item[\hskip \labelsep {\bfseries #1}\hskip \labelsep {\bfseries #2}]}{\end{trivlist}}

\title{The Geometry of Stable Quotients in Genus One}
\author{Yaim Cooper}

\def\C{\mathbb{C}}
\def\Q{\mathbb{Q}}
\def\Z{\mathbb{Z}}
\def\N{\mathbb{N}}
\def\R{\mathbb{R}}
\def\p{\mathbb{P}}
\def\Cstar{\mathbb{C}^*}

\def\OO{\mathcal{O}}

\def\pnmo{\p^{n-1}}

\def\Moo{M_{1,1}}
\def\Moobar{\overline{M}_{1,1}}
\def\Mdepsilon{\overline{M}_{1, d \cdot \epsilon}}
\def\Mdepsilon2{\overline{M}_{1, 0|d}}
\def\Mgmd{\overline{M}_{g,m|d}}

\def\qg1d{\mathrm{\overline{Q}}_{1}(G(1,1),d)}
\def\qgnd{\mathrm{\overline{Q}}_{1}(G(1,n),d)}
\def\qpnd{\mathrm{\overline{Q}}_{1}(\pnmo,d)}

\def\CC{\mathcal{C}}

\def\tilMoo{\widetilde{M}_{1,1}}
\def\tilMood{\widetilde{M}_{1,1}[d]}

\def\ocdp{\OO_\CC(d p)}

\def\RipiQ{R^i \pi_* \Q}
\def\RjpiQ{R^j \pi_* \Q}

\def\Zmod2{\Z/2}
\def\Zmodd{\Z/d}
\def\Ztwo{\Z/2}
\def\Zd{\Z/d}
\def\Zdsq{(\Z/d)^2}

\def\pdmo{\mathbb{P}^{d-1}}

\def\Pd{P_d}
\def\pd{P_d}

\def\tilPd{\widetilde{P_d}}

\def\pgln{PGL_n}
\def\sln{SL_n}

\def\pgl2{PGL_2}

\def\A{\mathbb{A}}

\def\Deld{{\Delta_d}}
\def\Deldi{{\Delta_d^i}}
\def\Deldid{{\Delta_d^i[d_1,...,d_i]}}
\def\Deld1{{\Delta_d^1}}

\def\Ndi{{N_{d-i}^i}}
\def\Nid{{N_i^{d-i}}}

\def\Cvi{C_{v_i}}
\def\Cei{C_{e_i}}

\def\M02Sd{\overline{M}_{0,2|d}/S_d}

\def\Pen{\mathcal{P}}
\def\Penab{\mathcal{P}_{a,b}}
\def\Penabprime{\mathcal{P'}_{a',b'}}

\def\Dfd{D_{fd}}
\def\Dj{D_j}
\def\Db{D_b}

\def\Dfd1{D_{fd}^1}
\def\Dj1{D_j^1}
\def\Db1{D_b^1}

\def\Dfdn{D_{fd}^n}
\def\Djn{D_j^n}
\def\Dbn{D_b^n}

\def\N{\mathcal{N}}
\def\NN{\mathfrak{N}}

\def\CC{\mathcal{C}}

\begin{document}
\date{}
\maketitle

\begin{abstract}
Stable quotient spaces provide an alternative to stable maps for compactifying spaces of maps.  When $n \geq 2$, the space $\overline{Q}_{g}(\p^{n-1},d) = \overline{Q}_{g}(G(1,n),d)$ compactifies the space of degree $d$ maps of smooth genus $g$ curves to $\p^{n-1}$, while $\overline{Q}_{g}(G(1,1),d) \simeq \overline{M}_{1, d \cdot \epsilon}/S_d$ is a quotient of a Hassett weighted pointed space.  In this paper we study the coarse moduli schemes associated to the smooth proper Deligne-Mumford stacks $\overline{Q}_{1}(\p^{n-1},d)$, for all $n \geq 1$.  We show these schemes are projective, rationally connected and have Picard number 2.  Then we give generators for the Picard group, compute the canonical divisor, and the cones of ample and effective divisors.  We conclude that $\overline{Q}_{1}(\p^{n-1},d)$ is Fano if and only if $n(d-1)(d+2) < 20$.  In the case $n=1$, we write in addition a closed formula for the Poincar\'{e} polynomial.
\end{abstract}

\section{Introduction} \label{sec:intro}

One of the foundational steps in the mathematical development of Gromov-Witten theory was Kontsevich's construction of moduli spaces of stable maps $\overline{M}_{g,m}(X, \beta)$ compactifying the space of maps of genus $g$ curves with $m$ marked points to a target $X$ with image in a fixed homology class $\beta$.  For the purposes of Gromov-Witten theory, the essential point is that although the moduli spaces are in general neither smooth nor of the expected dimension, they do carry a virtual class.

However, the geometry of the moduli spaces themselves is generally ill behaved.  The simplest target to consider is $\p^r$.  In this case, $\overline{M}_{0,m}(\p^r,d)$ is smooth (as a Deligne-Mumford stack), while for all other genera $g \geq 1$, the spaces $\overline{M}_{g,m}(\p^r,d)$ are generally singular and have multiple components of different dimensions.  This behavior begins almost immediately - even $\overline{M}_{1,0}(\p^2,3)$ compactifying the 9-dimensional classical moduli space of smooth plane cubics has 3 components, two of dimension 9 and one of dimension 10.  So for $g \geq 1$, it has been difficult to study the geometry of the moduli spaces $\overline{M}_{g,m}(\p^r,d)$.  In genus 0, Pandharipande gave generators for the Picard group, computed the canonical divisor, as well as intersections of divisors in $\overline{M}_{0,m}(\p^r,d)$ in \cite{Pand1} and \cite{Pand2}.

In 2009, Marian-Oprea-Pandharipande defined stable quotient spaces, giving a new compactification of the space of maps of genus $g$ curves to Grassmanians.  These stable quotient spaces tend to be more efficient compactifications than stable maps. As with stable maps, for all genera the moduli spaces of stable quotients carry a virtual class.  When the target is projective space, there is a map $c : \overline{M}_{g,m}(\pnmo,d) \twoheadrightarrow \overline{Q}_{g,m}(\pnmo,d)$ and the strongest possible comparison holds, that is $c_* [\overline{M}_{g,m}(\pnmo,d)]^{vir} = [\overline{Q}_{g,m}(\pnmo,d)]^{vir}$ \cite{MOP}.  For targets other than $\pnmo$, such a simple relationship is not expected.

One approach to probing the relationship between stable maps and stable quotients was proposed by Toda in \cite{Toda}.  There, he introduces a sequence of moduli spaces $\overline{Q}_{g,m}(G(r,n),d)^{\epsilon}$ $0 \leq \epsilon \in \R$ interpolating between $\overline{M}_{g,m}(G(r,n),d)$ and $\overline{Q}_{g,m}(G(r,n),d) = \overline{Q}_{g,m}(G(r,n),d)^0$.  For fixed $g,m,r,n,d$, there are finitely many spaces $\overline{Q}_{g,m}(G(r,n),d)^{\epsilon}$ in this sequence, and the change of moduli spaces as the stability parameter $\epsilon$ varies is an example of wall crossing.  When $r = 1$, whenever $\epsilon > \epsilon '$, there is a morphism $c^{\epsilon, \epsilon'}: \overline{Q}_{g,m}(G(r,n),d)^{\epsilon} \twoheadrightarrow \overline{Q}_{g,m}(G(r,n),d)^{\epsilon'}$.  The composition of all these maps is equal to the map $c$ above, and we have a finite sequence of spaces

$$\overline{M}_{g,m}(\pnmo,d) \twoheadrightarrow \overline{Q}_{g,m}(\pnmo,d)^{\epsilon_1} \twoheadrightarrow ... \twoheadrightarrow \overline{Q}_{g,m}(\pnmo,d).$$

Using these spaces it is possible to define a sequence of Gromov-Witten like invariants for a target variety $X \subset \pnmo$.  It is hoped that wall crossing formulae can be derived to understand how these invariants change as $\epsilon$ changes.

One feature of stable quotients is that in genus 1 when the target is projective space, the moduli spaces $\qpnd$ compactifying spaces of maps without marked points are smooth, and we can study their geometry directly.  This in turn may help to compute genus 1 stable quotient invariants (albeit invariants that don't require markings), and as discussed above, in turn may help to compute genus 1 Gromov-Witten invariants.  Previously, there have been other constructions of smooth compactifications of genus 1 maps.  In \cite{VakilZinger}, Vakil and Zinger construct a smooth compactification $\widetilde{M}_{1,m}(\pnmo,d)$ by sequentially blowing up loci in the Kontsevich stable maps moduli space.  Zinger subsequently used these spaces to compute the genus 1 Gromov-Witten invariants for the quintic threefold, validating physicists' predictions.  The relation between the Vakil-Zinger compactification and stable quotients can be thought of in the following way.  We have maps

$$\widetilde{M}_{1,m}(\pnmo,d) \twoheadrightarrow \overline{M}^0_{1,m}(\pnmo,d) \twoheadrightarrow \overline{Q}_{1,m}(\pnmo,d) $$

Where $\overline{M}^0_{1,m}(\pnmo,d)$ denotes the main component of the moduli space of stable maps.  The Vakil-Zinger space is in this sense larger than stable maps, while stable quotients are smaller, and in the case $m=0$ here, smooth.  The Vakil-Zinger spaces have the advantage of being smooth even with marked points.

Another feature of stable quotients is that they are easier to work with by localization than stable maps, for the simple reason that they have fewer $\Cstar$ fixed loci relative to the stable maps moduli spaces.  This makes localization computations more feasible on these spaces, and this was used by Pandharipande-Pixton in \cite{PandPix} to compute relations in the tautological ring of $M_g$, in particular proving that the conjectured Faber-Zagier relations in fact hold in $R^*(M_g)$.

Another similar application is the use by Pandharipande of stable quotient spaces to obtain relations in the $\kappa$ ring of the moduli of curves of compact type, $\kappa^*(M_{g,n}^{c})$ \cite{Pandkappa1}, \cite{Pandkappa2}.  Doing so, he is able to give generators and relations for these rings, as well as formulas for the Betti numbers.

Our goal in this paper is to understand the basic features of the geometry of the genus 1 spaces $\qpnd$.  $\qpnd$ is a smooth Deligne-Mumford stack of dimension $nd$ \cite{MOP}.  Knowing this, natural questions to address are, what topological information can be computed?  The cohomology ring?  Betti numbers?  Geometrically, we can ask for all the analogues to the known results in genus 0 about $\overline{M}_{0,n}(\pnmo, d)$.  What is the Picard group and canonical divisor?  (Computed in genus 0 by Pandharipande in \cite{Pand1} and \cite{Pand2}.)  What are the cones of ample and effective divisors?  (Computed in genus 0 by Coskun-Harris-Starr in \cite{CHS1} and \cite{CHS2}.  Is it possible to run Mori's program on this moduli space?  (Addressed in genus 0 by Chen-Chrissman-Coskun in \cite{CCC}.)

We address most of these questions for the case of the moduli spaces $\qpnd$.  The main tools we use are the virtual Poincare polynomial, a Bia{\l}ynicki-Birula stratification when $n \geq 2$, and standard intersection theory.

The main results contained in this paper are the following:

First the Poincar\'{e} polynomial of $\qg1d$ is computed, and the expected symmetry of Poincar\'{e} duality is seen directly in the proof.

\begin{nonumthm} {\ref{PoincareOfG11Formula}}

\text{If $d$ is odd,}

\begin{equation*}
P_{\qg1d} (t) =
\frac{1}{2d} \left[ d (1+t^2)(1+t^4)^{\frac{d-1}{2}} + \sum_{k|d}
\phi(\frac{d}{k}) (1 + t^{\frac{2d}{k}})^{k} \right] + \sum_{j=1}^{d-1} t^{2j}
\end{equation*}

and if $d$ is even,
\begin{equation*}
P_{\qg1d} (t) = \frac{1}{2d} \left[ \frac{d}{2} (1+t^2)^2(1+t^4)^{\frac{d}{2}-2} +
\frac{d}{2} (1+t^4)^{\frac{d}{2}} + \sum_{k|d} \phi(\frac{d}{k}) (1 + t^{\frac{2d}{k}})^k \right] + \sum_{j=1}^{d-1} t^{2j}.
\end{equation*}

\end{nonumthm}

Using the Bia{\l}ynicki-Birula stratification for the moduli space, we are able to in certain cases compute the Poincare polynomial of $\qpnd$, and generally, to show that:

\begin{nonumprop} {\ref{oddratcoh}}
For all $n, d$, and all $i$ odd, $h^i(\qpnd, \Q)=0$.  That is, the odd rational cohomology of $\qpnd$ vanishes.
\end{nonumprop}

and that

\begin{nonumthm}{\ref{PicardRank}}
The Picard rank of $\qpnd$ is equal to $h^2(\qpnd,\Q) = 2$.
\end{nonumthm}

The Picard rank being 2 greatly simplifies divisor computations on these moduli spaces, and we are able to describe much of the divisor theory of the moduli spaces $\qpnd$.  The divisors $D_j, D_{fd},$ and $D_b$ are defined in Section \ref{sec:ConstructionOfDivisors}.

\begin{nonumthm}{\ref{NefCone}}
The nef cone of $\qpnd$ is bounded by the divisors $D_j$ and $D_{fd}$.
\end{nonumthm}

\begin{nonumcor}{\ref{Projective}}
$\qpnd$ is a projective scheme.
\end{nonumcor}

\begin{nonumthm}{\ref{EffCone}}
The effective cone of $\qpnd$ is bounded by the divisors $D_j$ and $D_b$.
\end{nonumthm}

\begin{nonumthm}{\ref{Canonical}}
$K_{\qpnd}  =  \frac{\left( d-11 + (d-1)(n-1) \right)}{12} D_j - n D_b.$
\end{nonumthm}

Finally, by direct analysis, we prove that

\begin{nonumthm} {\ref{thm:ratconn}}
$\qpnd$ is rationally connected.
\end{nonumthm}

These results leave the following natural question unresolved - are Toda's $\epsilon$-stable quotients related to running the minimal model program on the Kontseivich spaces $\overline{M}_{1,0}(\pnmo,d)$?  More precisely, if one were to run the minimal model program on the main component of $\overline{M}_{1,0}(\pnmo,d)$, could the sequence of main components of Toda's spaces, culminating in $\qpnd$, arise as one of the possible outcomes?  The main component of $\overline{M}_{1,0}(\pnmo,d)$ is normal \cite{HuLi}, \cite{Zinger}, so it is possible to consider running the minimal model program on it.  Being rationally connected by Theorem \ref{thm:ratconn}, the last space in such a sequence should be a Fano fiber space.  So the first step in this line of questioning would be to check if $\qpnd$ is minimal in the sense that it admits a map to a smaller dimensional space $Y$ such that all fibers are Fano.

At the moment, what we can say is that there is a projection $\pi: \qpnd \rightarrow \Moobar$, and all but one fiber are known to be Fano.  Specifically, the fiber over any smooth genus 1 curve is isomorphic to $\p^{nd-1}/G$ for a finite group $G$, hence Fano.  The remaining fiber, over the nodal $\p^1$, is not normal, and it is not known if the normalization of that fiber is Fano or not.

The author would like to thank the following people.  My advisor R. Pandharipande for patiently teaching me the techniques used in this paper.  I. Coskun for the course "The birational geometry of the moduli spaces of curves" he gave at School on birational geometry and moduli spaces June 1-11, 2010 at the University of Utah during which I was inspired to consider the questions about the cones of nef and effective divisors.  Also, O. Biesel, D. Chen, A. Deopurkar, M. Fedorchuck, C. Fontanari, J. Li, A. Patel, D. Ross, V. Shende, D. Smyth, R. Vakil, M. Viscardi, and A. Zinger for helpful conversations.  The author was supported by an NSF graduate fellowship.

\subsubsection{Definition of Stable Quotients}

In this paper we work over $\C$, and by \textit{curve} will mean a reduced connected scheme of pure dimension 1.  We begin by recalling the definition of stable quotients, as introduced in \cite{MOP}.  In \cite{MOP} a more general definition allowing marked points is given, but as we will be concerned only with the unmarked case, we state only that definition here.

Let $C$ be a curve with at worst nodal singularities, of arithmetic genus $g$. A quotient $q$ of the trivial sheaf $$0 \rightarrow S \rightarrow \C^n \otimes \OO_C \stackrel{q}{\rightarrow} Q \rightarrow 0$$ is a \textit{quasi-stable quotient} if $Q$ is locally free at the nodes of $C$.  Quasi-stability implies that

(i) the torsion subsheaf $\tau(Q) \subset Q$ has support contained in $C^{ns}$, the nonsingular locus of $C$, and

(ii) $S$ is a locally free sheaf on $C$.

Let $k$ denote the rank of $S$.  Given $C$ and a quasi-stable quotient $Q$ on $C$, the data $(C,q)$ determines a \textit{stable quotient} if the $\Q$-line bundle
\begin{equation}
\label{stability}
\omega_C \otimes (\wedge^k S^*)^{\otimes \epsilon} \mbox{ is ample for all } 0<\epsilon \in \Q.
\end{equation}

No amount of positivity of $\bigwedge^k S^*$ can stabilize a genus 0 component unless it contains at least 2 nodes, and if such a component has exactly 2 nodes, then $\bigwedge^k S^*$ must have positive degree on it.

Two quasi-stable quotients $$0 \rightarrow S \rightarrow \C^n \otimes \OO_C \stackrel{q}{\rightarrow} Q \rightarrow 0 \hspace{11mm} \mbox{ and } \hspace{11mm} 0 \rightarrow S' \rightarrow \C^n \otimes \OO_C \stackrel{q'}{\rightarrow} Q' \rightarrow 0$$ are \textit{strongly isomorphic} if $S, S' \subset \C^n \otimes \OO_c$ are equal.  An \textit{isomorphism} of quasi-stable quotients $$\phi : (C,q) \rightarrow (C',q')$$ is an isomorphism of curves $\phi : C \stackrel{\sim}{\rightarrow} C'$ such that the quotients $q$ and $\phi^*(q')$ are strongly isomorphic.

In \cite{MOP}, Marian-Oprea-Pandharipande prove the following:

\begin{thm} [\cite{MOP} Theorem 1]
The moduli space of stable quotients $\overline{Q}_{g}(G(k,n),d)$ parameterizing the data $$(C, 0 \rightarrow S \rightarrow \C^n \otimes \OO_C \stackrel{q}{\rightarrow} Q \rightarrow 0)$$ with $rank(S) = k$ and $deg(S) = -d$ is a separated and proper Deligne-Mumford stack of finite type over $\C$.
\end{thm}

\begin{thm}[\cite{MOP} Proposition 1]
$\qgnd = \qpnd$ is a smooth irreducible Deligne-Mumford stack of dimension $nd$ for $d>0$.
\end{thm}

\subsection{Notation and conventions}\label{sec:notation}

The stability condition (\ref{stability}) constrains the isomorphism type of the underlying curve $C$.  In genus 1, there are only two possibilities.  If $(C,q)$ is a stable quotient, either $C$ is a smooth genus 1 curve, or $C$ is isomorphic to $C_i$, a cycle of $i$ rational curves, and $S^*$ has positive degree on each component of $C_i$.  Thus for a stable quotient in $\qpnd$ the underlying curve $C$ is either a smooth genus 1 curve or is isomorphic to $C_i$ for $1 \leq i \leq d$.  This gives a natural stratification on $\qpnd$.

\begin{figure}[b]
  \centering
  \includegraphics[width=.5in]{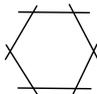}
  \caption{For example, the curve $C_6$.}
\end{figure}

Following \cite{MOP},

\begin{defn}
Let $\Mgmd$ be the moduli space parameterizing genus $g$ curves with markings $\{ p_1,...,p_m\} \bigcup \{q_1, ..., q_d\} \in C^{ns} \subset C$ satisfying

(i) the points $p_i$ are distinct,

(ii) the points $q_j$ are distinct from the points $p_i$

with stability given by the ampleness of $$\omega_C(\sum_{i=1}^m p_i + \epsilon \sum_{j=1}^d q_j)$$ for every strictly positive $\epsilon \in \Q$.
\end{defn}

Note that the points $q_j$ can collide.  $\Mgmd$ is a special case of Hassett's more general construction of weighted pointed spaces of curves, namely where the points $p_i$ each have weight 1 and $q_j$ each have weight $\epsilon$, for $\epsilon < \frac{1}{d}$.

We end with an alternate description of the moduli space $\qgnd$.  Given a line bundle $S$ and an inclusion $0 \rightarrow S \rightarrow \C^n \otimes \OO_C$, the dual sequence is $\C^n \otimes \OO_C \rightarrow S^* \rightarrow 0$ and is equivalent to the data of $n$ sections $s_i \in H^0(S^*)$.  So $\qpnd$ can alternatively be described as parameterizing, up to isomorphism, tuples $(C,(s_0,...,s_{n-1}))$ where $S^*$ is some line bundle of degree $d$ and $s_i \in H^0(S^*)$, subject to the same stability condition (\ref{stability}).  In this notation, $(C,(s_0,...,s_{n-1})) \sim (C',(s'_0,...,s'_{n-1}))$ if and only if there exists $\lambda \in \C^*$ and $\phi: C \stackrel{\sim}{\rightarrow} C'$ such that $\lambda \cdot s'_i \circ \phi = s_i$, for all $i$.

In the case $n=1$ then, $\qg1d$ parameterizes pairs $(C,D)$, $C$ a curve of genus 1 and $D$ an effective divisor of degree $d$ on $C$, up to automorphisms of $C$.  Stability requires $D$ to have at least one point on any rational component with exactly two nodes, and $(C,D) \sim (C',D')$ if there is an isomorphism $\phi: C \rightarrow C'$ with $\phi(D) = D'$.  We conclude that $\qg1d$ is isomorphic to the quotient $\Mdepsilon2/S_d$.

\section{Topology of $\qg1d$} \label{sec:top1}
In this section we compute the Poincar\'{e} polynomial of $\qg1d$.  The main tool used is the virtual Poincar\'{e} polynomial, which is additive on locally closed stratifications.  The stratification we use here is the one introduced earlier, by isomorphism type of the underlying curve.  $\qg1d$ has $d+1$ strata: $U_d$, the strata where $C$ is a smooth genus 1 curve, and $\Delta_d^i$, the strata where $C \simeq C_i$, for $1 \leq i \leq d$.  Our task reduces to computing the virtual Poincar\'{e} polynomial of each stratum.  We begin by recording the definition and some basic properties of the virtual Poincar\'{e} polynomial.

\subsection{Introduction to the Virtual Poincar\'{e} Polynomial}
Given a compact K\"{a}hler manifold $M$ and any $U \subset M$ the complement of a normal crossing divisor, there is a mixed Hodge structure on the cohomology groups with compact support $H_c^k(U,\Q)$\cite{Deligne},\cite{Voisin}.  This mixed Hodge structure gives a weight filtration $W$ on these cohomology groups, and the coefficients of the virtual Poincar\'{e} polynomial are defined from the dimensions of these graded pieces

\begin{defn}
$$P^{vir}_U(t) := \sum_{m,k}(-1)^{k+m} \mathrm{dim}(Gr^m_W(H^k_c(U)) t^m.$$
\end{defn}

We collect here a number of properties of the virtual Poincar\'{e} polynomial.  The first two can be found in \cite{Fulton}.

\begin{itemize}
\item{If $X$ is a smooth complete orbifold, then the Poincar\'{e} polynomial and virtual Poincar\'{e} polynomial are equal, $P_X(t) = P^{vir}_X(t)$.}
\item{If $Z$ is a closed algebraic subset of $X$ and $U = X \setminus Z$, then $P_X(t) = P_U(t)+P_Z(t)$.  Hence $P^{vir}_X(t)$ is additive on locally closed strata.}
\end{itemize}

We give proofs for the remaining three.

\begin{lem} \label{virpp for fiberbundles}
Given a fiber bundle $\pi: X \rightarrow B$ whose local systems $\RipiQ$ are constant for all $i$, the virtual Poincar\'{e} polynomial is multiplicative, that is $P^{vir}_X(t) = P^{vir}_B(t) \times P^{vir}_F(t)$.
\end{lem}

\begin{proof}
The standard tool for computing the cohomology of the fiber bundle $X$ is the Leray spectral sequence $E^{i,j}_r \Rightarrow H^{i+j}(X,\Q)$.  Deligne showed \cite{Deligne} that for a projective fibration this spectral sequence degenerates at the $E_2$ term, where $E^{i,j}_2 = H^i(B, R^j \pi_! \Q)$.  By M. Saito's theory of mixed Hodge modules, the Leray spectral sequence for cohomology with compact supports, which also degenerates at the $E_2$ term, $$E^{i,j}_2 = H_c^i(B, R^j \pi_! \Q) \Rightarrow H_c^{i+j}(X,\Q)$$ is a spectral sequence of mixed Hodge structures \cite{Saito}.

Assuming the local systems $\RipiQ$ are trivial for all $i$, by an analogue of the universal coefficient theorem $H_c^i(B,R^j \pi_* \Q) = H_c^i(B, \Q) \otimes H_c^j(F,\Q)$, so

\begin{flushleft}
\begin{flalign*}
P^{vir}_X(t)  & :=  \sum_{m,k}(-1)^{k+m} \mathrm{dim}(Gr^m_W(H^k_c(X)) t^m\\
& =  \sum_{m,k}(-1)^{k+m} \mathrm{dim}(Gr^m_W(H_c^k(B,R^j \pi_* \Q)) t^m\\
& =  \sum_{a+b=m, i+j=k} (-1)^{i+j+a+b} \mathrm{dim}(Gr_W^a(H_c^i(B)) \otimes Gr_W^b(H_c^j(F)) t^{m}\\
& =  \left(\sum_{a,i} (-1)^{a+i} \mathrm{dim}\left(Gr_W^a (H_c^i(B,\Q))\right) t^a \right)  \left(\sum_{b,j} (-1)^{b+j} \mathrm{dim}(Gr_W^b \left(H_c^j(F,\Q))\right) t^b \right)\\
& =  P^{vir}_B(t) \times P^{vir}_F(t).
\end{flalign*}
\end{flushleft}

\end{proof}

\begin{lem} \label{PnModGp}
If a finite group $G$ acts holomorphically on $\p^r$, the virtual Poincar\'{e} polynomial of the orbifold $\p^r/G$ is $1 + t^2 + ... + t^{2r}$.
\end{lem}

\begin{proof}
If a finite group $G$ acts holomorphically on a smooth projective variety $X$, $H^j(X/G, \Q) = H^j(X,\Q)^G$.  For smooth orbifolds the virtual Poincar\'{e} polynomial and Poincar\'{e} polynomial are equal so it will suffice to compute the ranks of $H^j(\p^r, \Q)^G$.

For $j$ odd, the rank is trivially zero.  For $j$ even, $H^j(\p^r,\Q)$ is one-dimensional and generated by the class of a linear subspace $S$.  Let $$\alpha := \sum_{g \in G} g(S) $$ be the $G$-orbit of $S$.  By assumption each $g \in G$ acts holomorphically on $\p^r$, in particular is orientation preserving.  So there is no cancellation in this sum and $\alpha$ is nonzero and by construction in $H^j(\p^r,\Q)^G$.  We conclude that $H^j(\p^r,\Q)^G$ is one-dimensional for $j$ even.
\end{proof}

\begin{lem} \label{CnModGp}
If a finite group $G$ acts linearly on $\C^r$, the virtual Poincar\'{e} polynomial of the orbifold $\C^r/G$ is $t^{2r}$.
\end{lem}

\begin{proof}
If $G$ acts linearly on $\C^r$, this action can be extended to $\p^r$, and when this is done, the original $\C^r$ and the $\p^{r-1}$ at infinity are invariant subspaces of this action on $\p^r$.  Hence $$P^{vir}(\p^r/G) = P^{vir}(\C^r/G) + P^{vir}(\p^{r-1}/G),$$ the previous lemma applies to both projective spaces, and we have $$ 1 + t^2 + ... + t^{2r} = P^{vir}(\C^r/G) + 1 + t^2 + ... + t^{2r-2},$$ completing the proof.
\end{proof}

\subsection{Computation of the virtual Poincar\'{e} polynomial of the main stratum $U_d$} \label{openlocus}
\subsubsection{The cases $d=1$ and $d=2$}

For $d \geq 3$, the moduli space $M_{1,1}[d]$ of elliptic curves with full level $d$ structure will be central to our computation of the virtual Poincar\'e polynomial of $U_d$.  Because $M_{1,1}[d]$ is a fine moduli space only when $d \geq 3$, our general approach cannot be applied when $d=1$ and $d=2$.  We begin by computing the virtual Poincar\'{e} polynomials of $U_1$ and $U_2$ directly.

When $d=1$, $U_1 \simeq M_{1,1}$ so

\begin{equation} \label{d=1}
P^{vir}_{U_1} = t^2.
\end{equation}

\begin{figure}[t]
  \centering
  \includegraphics[width=4in]{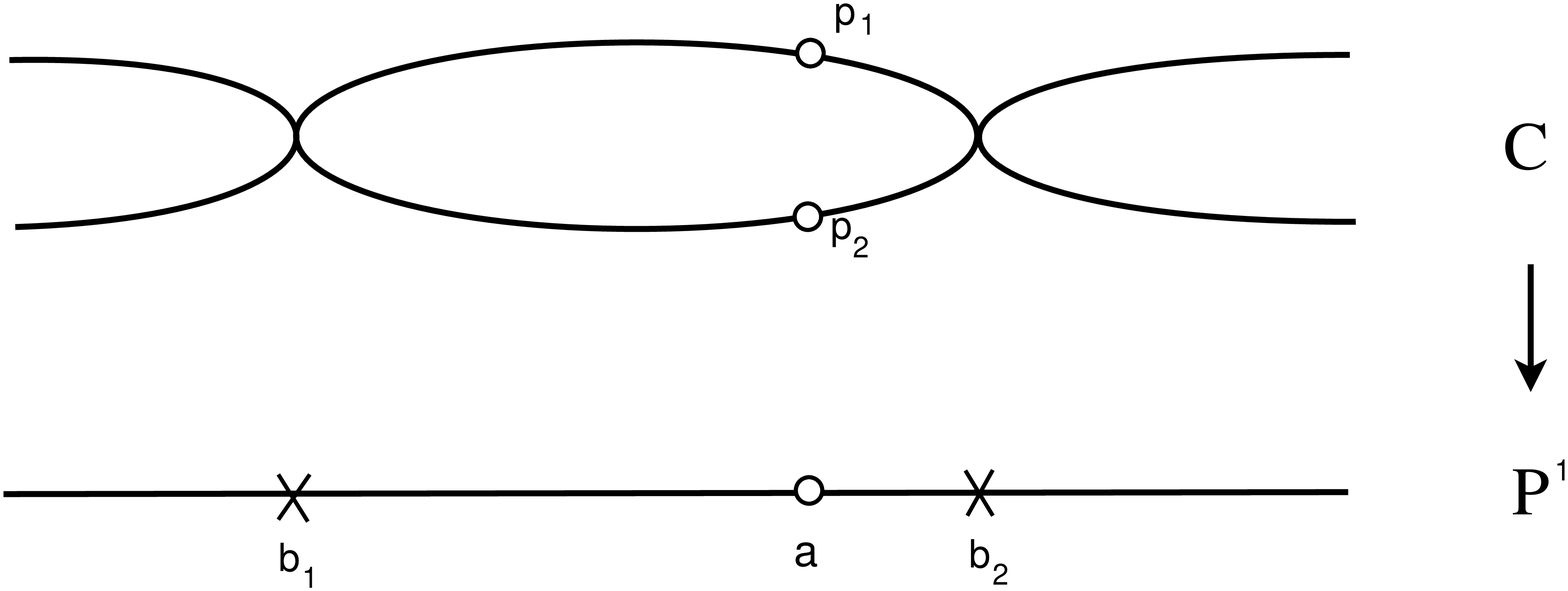}
  \caption{Map from $U_2$ to $M^{\langle 4,1 \rangle}.$}
\end{figure}

Let $M^{\langle 4,1 \rangle}$ denote the moduli space of $\p^1$ with five points $([b_1, b_2, b_3, b_4], a)$ on it, where the points $b_i$ are unlabelled and not permitted to collide with each other, and there is no constraint on the point $a$.  When $d=2$, $U_2 \simeq M^{\langle 4,1 \rangle}$ and the isomorphism $\Phi$ between these spaces is the following.  Take a point $(C, (p_1 + p_2)) \in U_2$.  The linear system $H^0(\OO_C(p_1+p_2))$ determines a double cover $f: C \rightarrow \p^1$.  Let $b_1,...,b_4$ be the distinct branch points of $f$ and $a = f(p_1) = f(p_2)$.  Then $\Phi (C, (p_1 + p_2)) = ([b_1, b_2, b_3, b_4], a)$.

To compute $P^{vir}_{U_2}(t)$, we use the following lemma, which is a special case of Theorem 5.4 of \cite{GetzlerPandharipande}.

\begin{lem} \label{modconnectedgp}
Let $G$ be a connected algebraic group and $X$ a quasi-projective variety a $G$-action.  If the action of G on X is almost-free, that is,
the isotropy groups of the action are finite, then $P^{vir}_X (t) = P^{vir}_G (t) P^{vir}_{X/G} (t).$
\end{lem}

Let $\mathfrak{M}^{\langle 4 \rangle}$ denote the parameter space of 4 distinct unlabelled points on a \textit{fixed} $\p^1$, and similarly $\mathfrak{M}^{\langle 4,1 \rangle}$ the space of 4 distinct unlabelled points and 1 unconstrained point on again, a \textit{fixed} $\p^1$.  $\pgl2$ acts naturally on both these spaces.  Note that $\mathfrak{M}^{\langle 4 \rangle}/\pgl2 \simeq \A^1$, because 4 distinct branch points on $\p^1$, up to $Aut(\p^1)$, determines a smooth $g=1$ curve, and conversely.  So

$$U_2 \simeq M^{\langle 4,1 \rangle} \simeq \mathfrak{M}^{\langle 4,1 \rangle}/\pgl2 \simeq (\p^1 \times \mathfrak{M}^{\langle 4 \rangle}) / \pgl2$$.

Applying Lemma \ref{modconnectedgp},

\begin{equation}\label{d=2}
P^{vir}_{U_2}(t) = P^{vir}_{(\p^1 \times \mathfrak{M}^{\langle 4 \rangle}) / \pgl2}(t) \\
= \frac{P^{vir}_{\p^1}(t) P^{vir}_{\mathfrak{M}^{\langle 4 \rangle}}(t)} {P^{vir}_{\pgl2} (t)} = P^{vir}_{\p^1}(t) P^{vir}_{\mathfrak{M}^{\langle 4 \rangle} / \pgl2} (t) \\
= P^{vir}_{\p^1}(t) P^{vir}_{\A^1}(t) = (1+t^2)(t^2).
\end{equation}

For the remainder of the discussion of $U_d$, we assume $d \geq 3$.

\subsubsection{Construction of $U_d$ as the quotient of a projective bundle}
To compute the virtual Poincar\'{e} polynomial of $U_d$, we give an explicit construction of $U_d$.  We begin with some background about moduli spaces of elliptic curves.  For all but two elliptic curves, the automorphism group of the curve is $\Ztwo$.  The exceptions are $\xi_4$ and $\xi_6$, with automorphism groups $\Z/4$ and $\Z/6$, respectively.  We will write $\widetilde{M}_{1,1} = \Moo \setminus \{\xi_4, \xi_6\}$, and generally for any moduli space $X$ of genus 1 curves and additional data, will denote by $\widetilde{X}$ the moduli space of a genus 1 curve other than $\xi_4$ or $\xi_6$, with the same additional data.

\begin{defn}
Fix $\zeta_d$ a $d^{th}$ root of unity.  $M_{1,1}[d]$ is the moduli space of elliptic curves with full level $d$ structure, and parameterizes the data of an elliptic curve $(C,p)$ with a basis $(e_1,e_2)$ of the $d$-torsion of $(C,p)$, such that the Weil pairing $\langle e_1,e_2 \rangle = \zeta_d$.
\end{defn}

$G = Sp_2(\Zd)$ acts on $\Moo[d]$ in the following way: for $g \in G$, $g \cdot (C,p,(e_1,e_2)) = (C, p, (g e_1, g e_2)).$  The subgroup $\langle -1,-1 \rangle \subset Sp_2(\Zd)$ acts trivially, and in fact $\Moo$ is the quotient $\rho: \Moo[d] \rightarrow \Moo[d]/(Sp_2(\Zd) / (\Z/2)) \simeq \Moo$.  Moreover, $Sp_2(\Zd) / (\Z/2)$ acts freely on $\tilMoo[d]$, so $\widetilde{\rho}: \tilMoo[d] \rightarrow \tilMoo$ is etale, and $\rho$ is ramified exactly over $\xi_4$ and $\xi_6$.

As we've taken $d \geq 3$, $\Moo[d]$ is a fine moduli space and carries a universal curve $\pi: \CC \rightarrow \Moo[d]$ together with a section $\sigma$.  By cohomology and base change $\pi_* \OO_\CC(d \sigma)$ is a vector bundle over $\Moo[d]$, because for any smooth genus 1 curve $C$ $H^1(\OO_C(dp)) = H^0(\omega_C(-dp)) = H^0(\OO_C(-dp) = 0$, so $R^1 \pi_* \OO_\CC(d \sigma) = 0$, while all higher $R^i \pi_* \OO_\CC(d \sigma)$ vanish on any curve.

The projectivization $P_d := \p(\pi_* \OO_\CC(d \sigma))$ is a $\pdmo$ bundle over $\Moo[d]$.  This is a fine moduli space parameterizing the data $(C, p, (e_1, e_2), [s])$ where $(C,p)$ is an elliptic curve, $(e_1,e_2)$ a basis for the $d$-torsion of that elliptic curve with fixed Weil pairing, $s \in H^0(\ocdp)$ a meromorphic function on $C$ with at worst a pole of order $d$ at $p$ (because $\ocdp$ comes with a canonical embedding in $K^*$), and $[s]$ the class of that function up to scaling by $\Cstar$.

There are two natural finite group actions on $P_d$.  The first is by $H := \Zdsq$.  To describe this action, note that on a smooth elliptic curve $(C,p)$, $H^0(\OO_C(dq-dp)) = 0$ unless $q$ is a $d$-torsion, in which case $\OO_C(dq-dp) \simeq \OO_C$ and $H^0(\OO_C(dq-dp)) = \C$.  Let $\tau_q$ denote the translation of $C$ sending $p$ to $q$.

Fix $(\alpha, \beta) \in H$.  Let $q = p + \alpha e_1 + \beta e_2$ and $\tau_{(\alpha, \beta)} = \tau_{q}$.  Pick a nonzero $f_{\alpha,\beta} \in H^0(\OO_C(dp-dq)) = \C$.  Define the action of $H$ on $\p(\pi_* \OO_\CC(d \sigma))$ by $(\alpha, \beta) \cdot (C, p, (e_1, e_2), [s]) = (C, p, (e_1, e_2), [\tau_{(\alpha, \beta)}^{-1} \circ (f_{\alpha,\beta} \cdot s)])$.  Because $f_{\alpha,\beta}$ was unique up to scaling, this action is well defined and $H$ acts fiberwise on the bundle $\pd$.

The second action is by $G  = Sp_2(\Zd)$.  Given $g \in G$, $g \cdot (C,p,(e_1,e_2), [s]) = (C, p, (g e_1, g e_2), [s])$.  Unlike with the action of $G$ on $\Moo[d]$, the subgroup $\langle -1,-1 \rangle$ does not act trivially on $\pd$.  Specifically, at any point $x \in \tilMoo[d]$ the isotropy group of this action is $\Ztwo$, while the isotropy groups at $(\xi_4,p)$ and $(\xi_6,p)$ are $\Z/4$ and $\Z/6$, respectively.  These isotropy groups all act nontrivially on the fibers of $\pd$, so in the quotient of $\pd$ by $G$, the fibers over $\xi_4$ and $\xi_6$ will be further quotients by $\Z/2$ and $\Z/3$, respectively.  On $\tilPd$, the action of $G$ can be considered in two parts - a $\Ztwo$ action that acts fiberwise and an equivariant, free $G/\Ztwo$ action.

Recall that $U_d$ can be described as the moduli space parameterizing $(C,D)$, $C$ a smooth genus 1 curve and $D$ an effective divisor of degree $d$, up to automorphisms of $C$.  There is a forgetful map $\Phi: P_d \rightarrow U_d$ defined on the moduli functors, $\Phi(C,p,(e_1,e_2),[s]) = (C, |s|)$, where $|s|$ is the zero set of the section $s$ of $\ocdp$.  This induces a map of the schemes $P_d \rightarrow U_d$.

\begin{prop}
$U_d = \Pd / H / G$.
\end{prop}

\begin{proof}
By Lemma \ref{phisurj} $\Phi$ is surjective, and by Lemma \ref{GHorbit} the preimage of every point in $U_d$ is a $G \times H$ orbit.
\end{proof}

\begin{lem} \label{phisurj}
$\Phi$ is surjective.
\end{lem}

\begin{proof}
Fix a point $(C,D) \in U_d$.  To show that $\Phi$ is surjective it will suffice to produce a point $p \in C$ and a meromorphic function $f \in H^0(\ocdp)$ such that $D$ is the divisor of zeroes of $f$ as a section of the line bundle $H^0(\ocdp)$.  Pick any point $z \in C$ to produce an elliptic curve $(C,z)$.  Suppose $D = p_1 + ... + p_d$.  Using the group law on the elliptic curve, we can take the sum $q = p_1 + ... + p_d$, and then pick a point $p$ such that $d p = q$.  Finally, pick any basis $(e_1,e_2)$ of the $d$-torsion of $(C,p)$ with $\langle e_1, e_2 \rangle = \zeta_d$.  Then $\Phi(C,p,(e_1,e_2),[f]) = (C,D)$.
\end{proof}

\begin{lem} \label{GHorbit}
The preimage of each point in $U_d$ under $\Phi$ is a $G \times H$ orbit.
\end{lem}

\begin{proof}
Fix a point $(C,D) \in U_d$, and suppose $(C, p, (e_1,e_2),[f]), (C, p', (e'_1,e'_2),[f'])$ are both in $\Phi^{-1}(C,D)$.  $\OO(D) = \OO(dp) = \OO(dp')$ so $p'$ is a $d$-torsion point of the elliptic curve $(C,p)$.  There is a unique $h = (\alpha, \beta) \in \Zdsq$ such that $p' = p + \alpha e_1 + \beta e_2$.  There is a unique $g \in Sp_2(\Zd)$ such that $(e'_1,e'_2) = (g \tau_{\alpha,\beta} e_1, g \tau_{\alpha,\beta} e_2)$.  It remains to show that $$g \cdot h (C, p, (e_1,e_2),[f]) := (C, \tau_{\alpha,\beta} p, (g \tau_{\alpha,\beta} e_1, g \tau_{\alpha,\beta} e_2), [\tau_{\alpha,\beta}^* f]) = (C, p', (e'_1,e'_2),[f']).$$

All equalities but $[\tau_{\alpha,\beta}^* f] = [f']$ are by construction, and the last is true because the isomorphism $H^0(\OO(dp))$ to $H^0(\OO(dp')$ is exactly $\tau_{\alpha,\beta}^*$.

\end{proof}

$G$ acts on $\widetilde{\Pd}$ by $\Ztwo$ on the fibers and the quotient $G/(\Ztwo)$ acts equivariantly.  So we can study $\widetilde{\Pd}/H/G$  by first taking the quotient $\widetilde{\Pd}/\hat{H}$, where $\hat{H} := H \times \Ztwo$ acts fiberwise, and then taking the further quotient by $\hat{G}:=G/(\Ztwo).$

$
\begin{CD}
\widetilde{P_d} \\
@VVV  \\
\widetilde{P_d}/\hat{H} @>>> \widetilde{P_d}/\hat{H}/\hat{G} \simeq \widetilde{U_d}\\
@VVV @VVV \\
\tilMoo[d] @>>> \tilMoo
\end{CD}
$

In the following, we will call $X$ a fiber bundle over base $B$ if $X$ is an algebraic variety mapping to $B$ and the map $f: X \rightarrow B$ is locally trivial in the etale topology.

\begin{prop} \label{Ud is fiber bundle}
$\widetilde{U}_d = \widetilde{P_d}/H/G$ is a fiber bundle over $\tilMoo$, with fibers $\pdmo/\hat{H}$.
\end{prop}

\begin{proof}
Since $\widetilde{\Pd}$ is the projectivization of a vector bundle over $\tilMoo[d]$, it is Zariski locally trivial, which implies etale locally trivial.  By Lemma \ref{StraightenOut} $\widetilde{\Pd}/\hat{H}$ is a $\pdmo / \hat{H}$ bundle over $\tilMoo[d]$.  Since $\hat{G}$ acts freely and equivariantly on $\widetilde{\Pd}/\hat{H}$ over $\tilMoo[d]$, the bundle descends to a fiber bundle $\widetilde{\Pd}/H/G$ over $\tilMoo$ with fibers $\pdmo /\hat{H}$.
\end{proof}

\begin{lem} \label{StraightenOut}
Given a fiber bundle $X \rightarrow B$ with fibers $\p^n$ and an action of a finite group $H$ on $X$ such that $H$ acts fiberwise, the quotient $X/H$ is a fiber bundle over $B$.
\end{lem}

\begin{proof}
It will suffice to give local trivializations respecting the $H$ action.  Start with any local trivialization of the fiber bundle $\p^n \times U \rightarrow U$, $b \in U$.  We can consider the trivial families of algebraic groups $H \times U$, $\pgln \times U$, and $\sln \times U$.  The action of $H$ on $X|_U$ determines a map $H \times U \rightarrow \pgln \times U$, and we can take the fiber product

$
\begin{CD}
\mathcal{H} @>>> \sln \times U\\
@VVV @VVV \\
H \times U @>>> \pgln \times U
\end{CD}
$

Each fiber of $\mathcal{H}$ over $U$ will be an extension of $H$ by $\Z/n$, and since the moduli of finite groups is discrete, they will all be isomorphic to a fixed group $\tilde{H}$.  After possibly an etale base change to $V$, we can trivialize $\mathcal{H}$ to obtain

$
\begin{CD}
\tilde{H} \times V @>>> \sln \times V\\
@VVV @VVV \\
H \times V @>>> \pgln \times V
\end{CD}
$

So we have an action of $\tilde{H}$ on the bundle $Y = \C^n \times V \rightarrow V$, the representation $R$ on each fiber isomorphic because taking the character on each fiber is a continuous function on $V$.  Let $Z$ be the bundle $\C^n \times V \rightarrow V$ with $\tilde{H}$ acting on $Z$ via the representation $R$ on the first factor.  By Schur's lemma the vector bundle Hom$(Z,Y)$ is a line bundle on $V$.  Take a nonzero section, and let $W$ be the open set on which it does not vanish.  Over $W$, $Z$ gives a trivialization of $Y$ for which the $\tilde{H}$ action is via only the first factor.  Projectivizing, we obtain a trivialization of $X$ over $W$ on which $H$ acts via only the first factor, as desired.
\end{proof}

\subsubsection{Computation of the virtual Poincar\'{e} polynomial of $U_d$}

\begin{lem}
$\RjpiQ_{\widetilde{U}_d}$ is a trivial local system on $\tilMoo$.
\end{lem}

\begin{proof}
We will show that $\RjpiQ_{\widetilde{U}_d}$ is trivial by demonstrating a nonvanishing global section.  A nonvanishing global section of $\RjpiQ_{\widetilde{U}_d} = \RjpiQ_{\tilPd/\hat{H}/\hat{G}}$ is simply a nonvanishing $\hat{H}$-invariant $\hat{G}$-equivariant global section of $\RjpiQ_{\tilPd}$.

$\tilPd$ is the projectivization of a vector bundle over $\tilMood$, hence is Zariski locally trivial.  So we can cover $\tilPd$ by open sets $\Omega$ over which there is a trivialization $\tilPd |_{\Omega} = \Omega \times \pdmo$.  For any such $\Omega$, fix a linear subspace $\Lambda \subset \pdmo$ of dimension $j$.  Under the trivialization, this gives a subvariety $\lambda := \Lambda \times \Omega$ of $\tilPd |_{\Omega}$.  The union $\lambda^{\hat{H}} := \sum_{h \in \hat{H}} h \cdot \lambda$ is $\hat{H}$-invariant, and its class $[\lambda^{\hat{H}}] \in H^j(P_d|_{\Omega})$ is $\hat{H}$-invariant and $\hat{G}$-equivariant.  Because the action of $\hat{H}$ is holomorphic on $\pdmo$, it is in particular orientation preserving and hence for each point $t \in \Omega$, the class $[\lambda^{\hat{H}}(t)]$ is a nonzero element of $H^j(\pdmo, \Q)^{\hat{H}}$.

In this way we obtain a section $s_{\Omega} \in \RjpiQ_{\tilPd} (\Omega)$ for each $\Omega$ in the cover.  For $\Omega$, $\Omega'$ two open sets of the cover, the sections $s_{\Omega}$ and $s_{\Omega'}$ agree on the intersection, so this defines a nonvanishing $\hat{H}$-invariant $\hat{G}$-equivariant global section of $\RjpiQ_{\tilPd}$, as desired.
\end{proof}

\begin{prop} \label{d>2}
$P^{vir}_{U_d}(t) = (1 + ... + t^{2d-2}) \cdot P^{vir}_{M_{1,1}}(t) = (t^2 + ... + t^{2d})$.
\end{prop}

\begin{proof}
By the additivity of the virtual Poincar\'{e} polynomial on locally closed stratifications,

$$P^{vir}_{U_d} (t) = P^{vir}_{\widetilde{U_d}} (t) + P^{vir}_{U_d|_{\Moo \setminus \tilMoo}} (t).$$

Because we've shown that the local system $\RjpiQ_{\widetilde{U}_d}$ is trivial, we can apply Lemma \ref{virpp for fiberbundles} to see that $P^{vir}_{\widetilde{U_d}} (t) = P^{vir}_{{\mathbb{P}}^{d-1}/\hat{H}}(t) \cdot P^{vir}_{\tilMoo}(t)$, and by Lemma \ref{PnModGp}, this is $(1 + ... + t^{2d-2})  P^{vir}_{\tilMoo}(t)$.

Meanwhile, the complement $\Moo \setminus \tilMoo$ consists of two points $\xi_4$ and $\xi_6$, and $U_d|_{\xi_4}$ (\textit{resp. $U_d|_{\xi_6}$}) is ${\mathbb{P}}^{d-1}/\hat{H}$ further quotiented by $\Z/2$ (resp. $\Z/3$).  In any case, by Lemma \ref{PnModGp}, the virtual Poincar\'{e} polynomial of these fibers is also $(1 + ... + t^{2d-2})$, and we conclude that

$$P^{vir}_{U_d} = (1 + ... + t^{2d-2})  P^{vir}_{\tilMoo}(t) + (1 + ... + t^{2d-2})  P^{vir}_{\Moo \setminus \tilMoo}(t) = (1 + ... + t^{2d-2})  P^{vir}_{\Moo}(t)$$

as desired.
\end{proof}

Note that Proposition \ref{d>2} holds for all $d$, as the expressions for $d=1$ and $d=2$ given in Equations (\ref{d=1}) and (\ref{d=2}) match this formula.

\subsection{Boundary strata}
Now we compute the virtual Poincar\'{e} polynomial of the strata $\Deldi$.  Recall that $\Deldi$ is the moduli space parameterizing sets $D$ of $d$ unordered possibly coincident points on $C_i^{ns}$, the nonsingular locus of $C_i$.  $D$ is isomorphic to $D'$ if there is an automorphism $\phi: C_i \rightarrow C_i$ such that $\phi(D) = D'$.  Because we work with objects modulo isomorphisms, in the remainder let us consider $C_i$ to be rigidified by the requirement that the two nodes on each component be $\{0,\infty\} \in \p^1$, and one component meets the next by gluing $0$ to $\infty$.  With this data fixed, the automorphism group of $C_i$ is $(\Cstar)^i \rtimes D_i$, where $(\Cstar)^i$ acts by scaling on each component and $D_i$ acts by the standard representation on an $i$-gon in the plane.

We further stratify $\Deldi$ into substrata defined by fixing, modulo the action of $D_i$, the number of points on each component of $C_i$.  Let $\Deldid$ denote the subset of $\Deldi$ with $D_1$ points on one component, $d_2$ points on the next, and so on.  We will show that $P_{\Deldid}^{vir}(t) = t^{2(d-i)}$ and count the number of such substrata.

\begin{lem}
$P_{\Deldid}^{vir}(t) = t^{2(d-i)}$.
\end{lem}

\begin{proof}
The simplest case is $\Deld1$.  Here $C_1$ is a nodal $\p^1$ and there is only one substratum $$\Deld1[d] \simeq Sym^d(\p^1 \setminus \{0,\infty\})/(\Cstar \rtimes \Zmod2).$$  Under the isomorphism $Sym^d(\Cstar) \simeq \C^{d-1} \times \Cstar$ given by $\{a_1,...,a_d\} \rightarrow (b_0,...,b_{d-1})$, where $z^d + b_{d-1}z^{d-1} + ... + b_0 = (z-a_1)...(z-a_d)$, the action of $\Cstar \rtimes \Zmod2$ is

$$(\lambda, \tau)(b_0,...,b_{d-1}) =
\left\{
\begin{array}{cl}
(\lambda^d b_0, ..., \lambda b_{d-1}) \mathrm{if  } \tau = id \in \Zmod2 \\
\\
(\frac{1}{\lambda^d b_0}, \frac{b_{d-1}}{\lambda^{d-1} b_0}, ... ,\frac{b_1}{\lambda b_0}), \mathrm{if  } \tau \neq id \in \Zmod2
\end{array}
\right.$$

So

$$
\Deldi[d] \simeq Sym^d(\Cstar)/(\Cstar \rtimes \Zmod2)
=(Sym^d(\Cstar)/\Cstar)/(\Zmod2) = (\C^{d-1}/\Zmodd)/(\Zmod2)$$

and we conclude that $P_{\Deld1[d]}^{vir}(t) = t^{2(d-1)}$ by Lemma \ref{CnModGp}.

In the general case,

$$\Deldid \simeq \left( Sym^{d_1}(\Cstar) \times ... \times Sym^{d_i}(\Cstar) \right) /(\Cstar)^i \rtimes D_i$$
$$= \left( Sym^{d_1}(\Cstar)/\Cstar \times ... \times Sym^{d_i}(\Cstar)/\Cstar \right) / D_i$$
$$= \left( Sym^{d_1-1}(\Cstar)/(\Z/d_1) \times ... \times Sym^{d_i-1}(\Cstar)/(\Z/d_i) \right) / D_i.$$

Again, by Lemma \ref{CnModGp} we conclude $P_{\Deldid}^{vir}(t) = t^{2(d-i)}$.
\end{proof}

It remains now only to compute the number of substrata of $\Deldi$.  We do so by observing the following bijection between the substrata $\Deldid$ of $\Deldi$ and necklaces of $i$ black beads and $d-i$ white beads.  The substrata of $\Deldi$ are indexed by $i$-gons decorated with a choice of a number $d_j \geq 1$ for each side, $d_j$ denoting the number of points of $D$ on the corresponding component of $C_i$.  Construct a necklace from such a decorated $i$-gon by going putting one black bead at every vertex and $d_j - 1$ white beads on each edge.

\begin{figure}[b]
  \centering
  \includegraphics[width=4in]{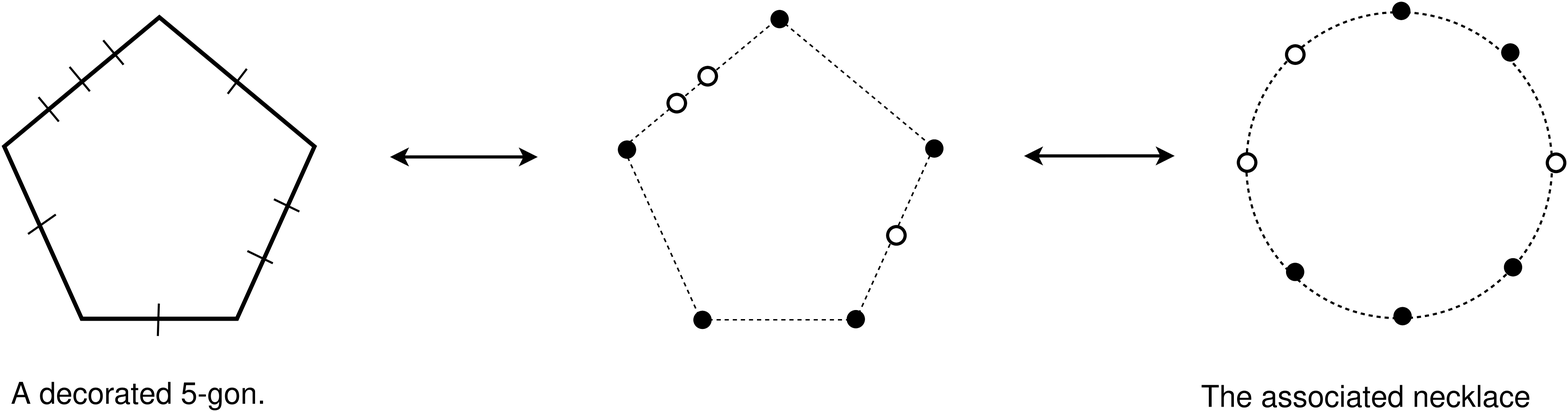}
  \caption{An example of the bijection between substrata $\Deldid$ and necklaces.}
\end{figure}

Let us denote by $\Ndi$ the number of necklaces of $i$ black beads and $d-i$ white beads.  We conclude
\begin{prop}
$P_{\coprod \Deldi}^{vir} (t) = \sum_{i = 1}^{d} \Ndi t^{2(d-i)}.$
\end{prop}

Combining this with the results of Section \ref{openlocus},

\begin{eqnarray*}\label{PoincareOfG11}
P_{\qg1d}^{vir} (t) = (t^2 + ... + t^{2d}) + \sum_{i = 1}^{d} \Ndi t^{2(d-i)}\\
\\
= (t^2 ... + t^{2d-2}) + \sum_{i = 0}^{d} \Ndi t^{2(d-i)}\\
\\
= 1 + 2t^2 + ... + 2t^{2d-2} + t^{2d}.
\end{eqnarray*}

Note that the necklace counting problem is symmetric under the interchange of the black and white beads.  That is, $\Ndi = \Nid$.  By Poincar\'{e} Duality, we know the Poincar\'{e} polynomial of $\qg1d$ will satisfy $$P_{\qg1d}^{vir} (t) = \left( P_{\qg1d}^{vir} (t) \right)^{-1} \cdot t^{2d}.$$  In this particular case, switching black and white beads is the explicit manifestation of this expected symmetry.

The question of counting necklaces of colored beads is a classical one in combinatorics.  Such necklaces are named "Polya's necklaces" in honor of Polya who gave a formula for the generating function for $\Ndi$ (also given independently by Redfield) \cite{Polya}, \cite{Redfield}.  Using his formula we conclude

\begin{thm} \label{PoincareOfG11Formula}

\text{If $d$ is odd,}

\begin{equation*}
P_{\qg1d} (t) = P_{\qg1d}^{vir} (t) = \{
\frac{1}{2d} \left[ d (1+t^2)(1+t^4)^{\frac{d-1}{2}} + \sum_{k|d}
\phi(\frac{d}{k}) (1 + t^{\frac{2d}{k}})^{k} \right] + \sum_{j=1}^{d-1} t^{2j}
\end{equation*}

and if $d$ is even,
\begin{equation*}
P_{\qg1d} (t) = P_{\qg1d}^{vir} (t) = \frac{1}{2d} \left[ \frac{d}{2} (1+t^2)^2(1+t^4)^{\frac{d}{2}-2} +
\frac{d}{2} (1+t^4)^{\frac{d}{2}} + \sum_{k|d} \phi(\frac{d}{k}) (1 + t^{\frac{2d}{k}})^k \right] + \sum_{j=1}^{d-1} t^{2j}
\end{equation*}

\end{thm}

\section{Topology of $\qpnd$} \label{sec:topn}
\subsection{Introduction to the Bia{\l}ynicki-Birula stratification}
Let $X$ be a smooth projective variety with a $\Cstar$ action.  The fixed locus $X^{\Cstar}$ is smooth and hence the disjoint union of irreducible components $X_1 \bigcup ... \bigcup X_r$ \cite{Iversen}.  Bia{\l}ynicki-Birula \cite{BB1} gave a natural stratification of $X$ indexed by fixed loci $X_i$, namely the cell $X_i^+$ associated to $X_i$ is defined as

$$X_i^+ = \{x \in X \hspace{2mm} | \hspace{2mm} \underset{\lambda \rightarrow 0}{lim} \lambda x \in X_i \}.$$

An alternate description is by considering the normal bundle $N_i$ to $X_i$.  $N_i$ decomposes into a direct sum $$N_i = \bigoplus_{m \in \Z} N_i(m)$$ where $N_i(m)$ is the subbundle of $N_i$ composed of semi-invariants of weight $m$.  Then $X_i^+$ is isomorphic to the total space of the subbundle $\oplus_{m>0} N_i(m)$.  Let $d_i(m) = \text{dim} N_i(m)$, $d_i^+ = \sum_{m>0} d_i(m)$.

\begin{thm} \label{BBthm} (\cite{BB1} Theorem 1)
If $X$ is a smooth projective variety with a $\Cstar$ action, then
$$P_X(t) = \sum_{i=1}^{r} P_{X_i}(t) t^{2d_i^+}.$$
\end{thm}

It is natural to extend this result to singular varieties \cite{CarrellGoresky}, and in fact Theorem \ref{BBthm} can be extended to the case that $X$ is a smooth Deligne-Mumford stack, giving a tool for computing the Betti numbers of such stacks.  Specifically,

\begin{thm} \label{BBforstacks} (\cite{Fontanari} Proposition 1.)
Let $X$ be a smooth projective orbifold with a $\Cstar$-action, let $F$ be the fixed locus and let $F_i$ denote its connected components.  Then

(i) $X$ is the disjoint union of locally closed subvarieties $S_i$ such that every $S_i$ retracts onto the corresponding $F_i$ and

$$\overline{S}_i \subseteq \bigcup_{j \geq i} S_j$$

(ii) the Betti numbers of $X$ are

$$h^m(X) = \sum_i h^{m-2n_i}(F_i)$$

where $n_i$ is the codimension of $S_i$.
\end{thm}

Note that for all Deligne-Mumford stacks with a coarse moduli scheme, the cohomology groups with $\Q$ coefficients of the stack and the scheme are equal.

\subsection{$\Cstar$ fixed loci of $\qpnd$} \label{sec:FixedLoci}
To use Theorem \ref{BBforstacks} to compute the Poincar\'{e} polynomial of $\qpnd$, we begin by fixing an action of $\Cstar$ on $\qpnd$ and describing the fixed loci of that action.  Fix a general $(w_o,...,w_{n-1}) \in \Z^n$ (in particular such that $w_i$ have no common factors) satisfying $w_0 < w_1 < ... < w_{n-1}$.  Let $\Cstar$ act on $\C^n$ via $$\lambda \cdot (x_0, ... , x_{n-1}) = (\lambda^{w_0} x_0, ... ,\lambda^{w_{n-1}} x_{n-1}).$$  This induces an action on $\qpnd$ by taking this action on $\C^n \otimes \OO_C$ of each stable quotient.

The fixed loci for this action on any stable quotient space $\mathrm{\overline{Q}}_{g}(Gr(k,n),d)$ are described in \cite{MOP}.  We record their result in the genus 1 rank 1 case here, with simplified notation.  The fixed loci of $\qgnd$ come in two types.

Type A.

There are $n$ fixed loci isomorphic to $\qg1d$.  They are $X_{\ell} := i_{\ell}(\qg1d)$ for $0 \leq \ell \leq n-1$, where the inclusion $i_{\ell}: \qg1d \rightarrow \qpnd$ is defined by

\begin{equation} \label{qg1d in qgnd}
i_{\ell} : \left( C, 0 \rightarrow S \rightarrow \OO \rightarrow Q \rightarrow 0 \right) \rightarrow \left( C, 0 \rightarrow S \rightarrow \begin{array}{l}
\OO^{\ell} \\
\oplus \\
\OO \\
\oplus \\
\OO^{n-\ell-1} \end{array} \rightarrow
\begin{array}{l}
\OO^{\ell} \\
\oplus\\
Q \\
\oplus\\
\OO^{n-\ell-1} \end{array} \rightarrow 0 \right).
\end{equation}

The normal bundle $N_{\ell}$ to $X_{\ell}$ is $Hom(S, \OO^{\ell} \oplus \OO^{n-\ell-1})$ and the subbundle of positive weight is $N_{\ell}^+ = Hom(S,\OO^{\ell})$, with dimension

\begin{equation} \label{typeAweights}
d_{\ell}^+ = d \ell.
\end{equation}

Type B.

These fixed loci are indexed by the following decorated graphs:

\begin{itemize}

\item{$\Gamma$: an $m$-cycle.  Let the vertices be labelled $\{v_i | 1 \leq i \leq m\}$ and the edges $\{e_i | 1 \leq i \leq m\}$ where $e_i$ is the edge from $v_i$ to $v_{i+1}$.}

\item{$\nu : \{ \text{vertices} \} \rightarrow \{ \Cstar \text{ fixed points in } \pnmo \}$}

\item{$s : \{ \text{vertices} \} \rightarrow \mathbb{N} \cup \{0\}$}

\item{$\delta : \{ \text{edges} \} \rightarrow \{ \text{positive integers} \}$}

\end{itemize}

subject to the constraints

\begin{equation}
\nu(v_i) \neq \nu(v_{i+1})
\end{equation}

with the indices taken modulo $m$, i.e. $\nu(v_m) \neq \nu(v_1)$, and

\begin{equation}
\sum_i s(v_i) + \delta(e_i) = d.
\end{equation}

To fix notation, name the $\Cstar$ fixed points in $\pnmo$ $p_0, ..., p_{n-1}$ where $p_j = [0:...:1:...:0]$ with all but the $j^{th}$ coordinate equal to 0. If $\nu(v_i) = p_j$, we will write $\nu_i = j$.  Similarly, let $s_i$ denote $s(v_i)$ and $\delta_i$ denote $\delta(e_i)$.  The value of $s$ on a vertex will record the amount of torsion on the curve associated to that vertex and the value of $\delta$ on an edge will be the covering number for the map from the associated $\p^1$.

\begin{figure}[t]
  \centering
  \includegraphics[width=2in]{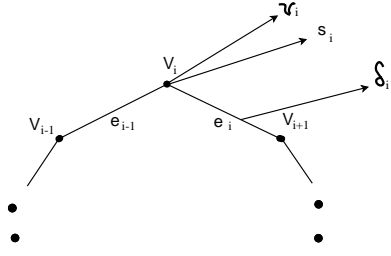}
  \caption{$\Gamma$ and its decorations}
\end{figure}

Fix such a decorated graph $\Gamma$.  Let $\mu$ be the number of vertices for which $s_i > 0$.  The points of the associated fixed locus $X_{\Gamma}$ are certain stable quotients on $C \simeq C_{m + \mu}$, the cycle of $m + \mu$ rational curves.  $C$ consists of components $\{ C_{v_i} \}$ for $i$ such that $s_i > 0$ and $\{ C_{e_i} \}$ for all $i$, glued by the graph incidences.  Let $C' \simeq C_{m}$ be the $m$-cycle of rational curves obtained by gluing only the components $\{ C_{e_i} \}$ along the graph incidences.

$X_{\Gamma}$ is isomorphic to $\displaystyle{\prod_{i | s_i > 0} \overline{M}_{0,2|s_i}} / Aut(\Gamma)$.  The isomorphism $\Psi: \displaystyle{\prod_{i | s_i > 0} \overline{M}_{0,2|s_i}} / Aut(\Gamma) \rightarrow X_{\Gamma}$ is given as follows.  Fix a point $\zeta \in \displaystyle{\prod_{i | s_i > 0} \overline{M}_{0,2|s_i}} / Aut(\Gamma)$.  For each $i$ such that $s_i>0$, $\zeta$ specifies marked points $\{p_1, p_2\} \bigcup \{ q_{1}...q_{s_i} \}$ on $C_{v_i} \simeq \p^1$.  The stable quotient on the component $C_{v_i}$ is $$0 \rightarrow \OO_{\Cvi} (-q_1 ... -q_{s_i}) \rightarrow \C^n \otimes \OO_{\Cvi} \rightarrow Q \rightarrow 0$$ where the inclusion is into the factor of $\C^n \otimes \OO_{\Cvi}$ determined by the $\Cstar$-fixed point of $\pnmo$ associated to $v_i$.  Namely, if $\nu(v_i) = p_j$ then $\OO_{\Cvi}(-q_1...-q_{s_i})$ injects into the $j^{\text{th}}$ copy of $\OO_{\Cvi}$.

The stable quotient on the component $C_{e_i}$ is obtained as follows.  Let $L \simeq \p^1$ be the line in $\pnmo$ between $\nu(v_i)$ and $\nu(v_{i+1})$.  Let
$$f_i: C_{e_i} \simeq \p^1 \rightarrow L$$
be the map of degree $\delta(e_i)$ ramified over completely over $\nu(v_i)$ and $\nu(v_{i+1})$.  The stable quotient on $C_{e_i}$ is obtained by pulling back the tautological sequence of $\pnmo$ along $f_i$.

By construction the stable quotients on the components $\Cvi$ and $\Cei$ are compatible so they can be glued to obtain a stable quotient on $C$, and this is the $\Cstar$-fixed point $\Psi(\zeta)$.

\subsection{Bia{\l}ynicki-Birula Cells}
To use the above $\Cstar$ action to compute the Poincar\`{e} polynomial of $\qpnd$, we must compute the fiber dimension of each cell.  To compute the fiber dimension $d_{\Gamma}^+$ of the Bia{\l}ynicki-Birula affine bundle over $X_{\Gamma}$, it suffices to take any point $x \in X_{\Gamma}$ and compute the weights of the $\Cstar$ action on the normal bundle $N_{\Gamma}$ to $X_{\Gamma}$ at $x$.  $d_{\Gamma}^+$ is the number of positive weights of this representation.

First consider the case that all $s_i=0$ in the decorated graph $\Gamma$.  Then the associated fixed locus is a point and the associated stable quotient is the pullback of the tautological sequence of $\pnmo$ along the map $f: C \rightarrow \pnmo$ described above, where $C \simeq C_m$ with components $\{\Cei\}$. In this case the deformation theory is just that for stable maps, as described in \cite{redbook}.  The normal bundle in this case is just the tangent space to the point, which in K-theory is
$$-Aut(C) + \bigoplus_i H^0(f_i^* T_{\pnmo}) - \bigoplus_i H^0((f^*T_{\pnmo})|v_i) + Def(C).$$
Note in this case the curves $C = C'$ are equal.

In the general case, Marian-Oprea-Pandharipande compute that for each nonzero $s_i$ the tangent space has an additional summand of
$$\oplus_{j \neq \nu(i)} H^0(\OO_C(\sigma_i)|\sigma_i) \otimes [w_j - w_{\nu(i)}]$$
where $\sigma_i \subset \Cvi$ is the divisor of torsion points corresponding to $s_i$ and $j$ runs from $0$ to $n-1$.  $H^0(\OO_C(\sigma_i)|\sigma_i)$ carries the trivial $\Cstar$ action and $[w_j - w_{\nu(i)}]$ denotes the vector space $\C$ with the representation $\lambda \cdot v = \lambda^{w_j - w_{\nu(i)}} v$ \cite{MOP}.

We conclude that in K-theory, the fiber of $N_{\Gamma}$ at $x$ is

\begin{equation}
-Aut(C') + \bigoplus_i H^0(f_i^* T_{\pnmo}) - \bigoplus_i H^0((f^*T_{\pnmo})|v_i) + Def(C) + \bigoplus_i \bigoplus_{j \neq \nu(i)} H^0(\OO_C(\sigma_i)|\sigma_i) \otimes [w_j - w_{\nu(i)}].
\end{equation}

Arbitrary subsums of this expression in K-theory will not necessarily correspond to vector spaces, but the following decomposition does correspond to a decomposition into three vector spaces: $N_{\Gamma}|_x = N_1 \bigoplus N_2 \bigoplus N_3,$ where

\begin{equation}
\begin{array}{l}
N_1 = \bigoplus_i N_1(i) = \bigoplus_i [ H^0(f_i^* T_{\pnmo}) - H^0((f^*T_{\pnmo})|v_{i+1}) - Aut(\Cei) ]\\
N_2 = \bigoplus_i N_2(i) = Def(C)\\
N_3 = \bigoplus_i N_3(i) = \bigoplus_i \bigoplus_{j \neq i} H^0(\OO_C(\sigma_i)|\sigma_i) \otimes [w_j - w_i].
\end{array}
\end{equation}

The $\Cstar$ weights of $N_1(i)$ are

\begin{equation} \label{N1}
\left\{
\begin{array}{c c}
\frac{c}{\delta_i} (w_{\nu_{i}} - w_{\nu_{i+1}}) & \quad \text{ for all } \begin{array}{l} -\delta_i \leq c \leq -1 \\ 1 \leq c \leq \delta_i -1 \end{array}\\
\\
\text{and}\\
\\
w_j - (\frac{\delta_i-c}{\delta_i} w_{\nu_i} + \frac{c}{\delta_i} w_{\nu_{i+1}}) & \quad \text{ for all } \begin{array}{l}j \neq \nu_i \text{ or } \nu_{i-1}\\ 0 \leq c \leq \delta_i -1.\end{array}
\end{array}
\right\}
\end{equation}

The vector space $N_2(i)$ is the deformation of $C$ given by smoothing the node $v_i$ if $s_i=0$ or smoothing the two nodes on $\Cvi$ if $s_i \neq 0$. The $\Cstar$ weights of $N_2(i)$ are thus

\begin{equation} \label{N2}
\left\{
\begin{array}{c l}
w_{\nu_{i+1}} + w_{\nu_{i-1}} - 2 w_{\nu_i} & \quad \text{ if } s_i = 0 \\
\\
\text{or}\\
\\
w_{\nu_{i+1}} - w_{\nu_i} \text{ and } w_{\nu_{i-1}} - w_{\nu_i} & \quad \text{ if } s_i \neq 0.
\end{array}
\right\}
\end{equation}

Finally, the $\Cstar$ weights of $N_3(i)$ are easiest to calculate, and are

\begin{equation} \label{N3}
\left\{
\begin{array}{c l}
w_j-w_i \text{ with multiplicity } s_i,  & \quad \text{ for all } j \neq i.
\end{array}
\right\}
\end{equation}

So the set of all weights of the normal bundle to $X_{\Gamma}$ are given by the expressions (\ref{N1}), (\ref{N2}), (\ref{N3}) as $i$ runs from 1 to $m$.  The dimension of the fibers of the affine bundle $X_{\Gamma}^+$ over $X_{\Gamma}$ is $d_{\Gamma}^+$, the number of positive weights.

\subsection{Computation of Betti numbers of $\qpnd$}
Marian-Oprea-Pandharipande proved that $\qpnd$ is a smooth Deligne-Mumford stack whose coarse moduli space is a scheme \cite{MOP}, thus we can apply Theorem \ref{BBforstacks} to compute the Betti numbers of $\qpnd$ by the formula $$P_X(t) = \sum_{i=1}^{r} P_{X_i}(t) t^{2d_i^+}.$$

For small $n,d$, we wrote down all the fixed loci and with the aid of some Mathematica script computed $d_{\Gamma}^+$ for all of them.  (For $n=d=2$, there was 1 fixed locus of type B and 4 weights to compute for that locus.  For $n=d=4$, there were $\sim 100$ fixed loci of type B and $\sim 1600$ weights.)  The result is:

\begin{equation}
\begin{array}{ccc}
P_{\mathrm{\overline{Q}}_{1}(\mathbb{P}^1,2)}(t) = 1 + 2t^2 + 3t^4 + 2t^6 + t^8\\
\\
P_{\mathrm{\overline{Q}}_{1}(\mathbb{P}^2,3)}(t) = 1 + 2t^2 + 3t^4 + 4t^6 + 4t^8 + 4t^{10} + 4t^{12} + 3t^{14} + 2t^{16} + t^{18}\\
\\
P_{\mathrm{\overline{Q}}_{1}(\mathbb{P}^3,4)}(t)= \\
\\
\text{\small
$1 + 2t^2 + 4t^4 + 5t^6 + 9t^8 + 10t^{10} + 14t^{12} + 14t^{14} + 17t^{16} + 14t^{18} + 14t^{20} + 10t^{22} + 9t^{24} + 5t^{26} + 4t^{28} + 2t^{30} + t^{32}$.}\\
\end{array}
\end{equation}

While it is difficult to compute the full Poincar\'{e} polynomial of $\qpnd$ for arbitrary $n,d$ because of the large number of fixed loci, it is nonetheless worthwhile to compute the most accessible Betti numbers, namely $h^i(\qpnd, \Q)$ for $i = 2$ and $i$ odd.

In the case of $h^2(\qpnd, \Q)$, the main observation is

\begin{lem}\label{TypeBd>2}
For any fixed locus $X_{\Gamma}$ in $\qpnd$ of type B, $d_{\Gamma}^+ \geq 2$.
\end{lem}

\begin{proof}
Fix a decorated graph $\Gamma$.  Pick a vertex $v_i$ of $\Gamma$ such that $\nu_i$ is minimal.  We demonstrate two positive weights of the bundle $N_{\Gamma}$.  The first is the weight $w_{\nu_{i+1}}-w_{\nu_i}$ of $N_1(i)$, which appears in the first expression of (\ref{N1}) for $c = \delta_i$.  This is positive because by our original choice of $\Cstar$ action, if $\nu_i$ is minimal then $w_{\nu_i} < w_{\nu_{i+1}}$.

The second is a weight of $N_2(i)$.  Here there are two cases.  If $s_i=0$, we have the weight $w_{\nu_{i+1}} + w_{\nu_{i-1}} - 2w_{\nu_{i}}$.  If $s_i>0$ we have two weights $w_{\nu_{i-1}} - w_{\nu_{i}}$ and $w_{\nu_{i+1}} - w_{\nu_{i}}$.  This gives either one or two positive weights, again by minimality of $w_{\nu_i}$.  In either case there are at least two positive weights of the normal bundle $N_{\Gamma}$, and we conclude $d_{\Gamma}^+ \geq 2$.
\end{proof}

\begin{prop}
$h^2(\qpnd, \Q)=2$ for all $n,d$.
\end{prop}

\begin{proof}
We consider the cases $d=1$ and $d \geq 2$ separately.  When $d=1$,
$$\mathrm{\overline{Q}_{1}(\pnmo,1)} \simeq \Moobar \times \pnmo \simeq \p^1 \times \pnmo (\mathrm{\cite{MOP} \ Section \ 2)}.$$
We see directly in this case that $h^2(\p^1 \times \pnmo) = 2$.

When $d \geq 2$:
For fixed loci of type A, which we'd named $X_{\ell}$ for $0 \leq \ell \leq n-1$, we've seen $d_{\ell}^+ > 2$ for all $\ell \geq 1$ (\ref{typeAweights}).  Together with Lemma \ref{TypeBd>2} this implies that in the expression

$$P_{\qpnd}(t) = \sum_{i} P_{X_i}(t) t^{2d_i^+}$$

given by Theorem \ref{BBforstacks} there is only one fixed locus that contributes to the $t^2$ term, namely $X_0 \simeq \qg1d$, because the Bialinicki-Birula cell associated to $X_0$ has zero dimensional fibers.  By Theorem \ref{PoincareOfG11Formula}, $P_{\qg1d} = 1 + 2t^2 +...+ 2t^{2d-2} + t^{2d}$.  Therefore $$h_2(\qpnd, \Q) = h_2(\qg1d) = 2.$$  By Poincar\'{e} duality, $h^2(\qpnd, \Q)$ also equals 2.
\end{proof}

To compute $h^i(\qpnd, \Q)$ when $i$ is odd, we will use a refinement of the following Lemma from \cite{MOP}.

\begin{lem} (\cite{MOP}, Lemma 2).  \label{Q02Bettis}
For $d>0$, the Poincar\'{e} polynomial of $\M02Sd$ is $$P_{\M02Sd}(t) = (1 + t^2)^{d-1}. $$
\end{lem}

\begin{prop} \label{oddratcoh}
For all $n, d$, and all $i$ odd, $h^i(\qpnd, \Q)=0$.  That is, the odd rational cohomology of $\qpnd$ vanishes.
\end{prop}

\begin{proof}
By Theorem \ref{BBforstacks} $P_{\qpnd}(t) = \sum P_{X_i}(t) t^{2d_i^+},$ so if the Poincar\'{e} polynomial of every fixed locus vanishes in odd degree, $P_{\qpnd}(t)$ will too.  Recall that there are two types of fixed loci.  A fixed locus $X_{\ell}$ of the first type is isomorphic to $\qg1d$, and by Theorem \ref{PoincareOfG11Formula} the Poincar\'{e} polynomial of $X_{\ell}$ vanishes in odd degree.

A fixed locus of the second type $X_{\Gamma}$ is isomorphic to some $$\displaystyle{\prod_{i | s_i > 0} \left( \overline{M}_{0,2|s_i} / S_{s_i} \right) } / Aut(\Gamma),$$ where we take the convention that the empty product is a point.  Lemma \ref{Q02Bettis}, does not compute the Poincar\'{e} polynomial directly if Aut$(\Gamma)$ is not trivial.  But in that case, a slightly modified version of the proof of Lemma \ref{Q02Bettis} of \cite{MOP} adding the use of Lemma \ref{CnModGp} gives

$$P^{vir}_{X_{\Gamma}} (t) = \prod_{i|s_i > 0} (1+t^2)^{s_i-1}.$$

We conclude that for all fixed loci the Poincar\'{e} polynomial vanishes in odd degree, so $P_{\qpnd}(t)$ also vanishes in odd degree.
\end{proof}

In \cite{Fontanari}, Fontanari proves a similar result for the Vakil-Zinger desingularization $\widetilde{M}^0_{1,n}(\p^r,d)$ of $\overline{M}_{1,n}(\p^r,d)$:

\begin{thm} (\cite{Fontanari}, Theorem 1.)
$H^k(\widetilde{M}^0_{1,n}(\p^r,d), \Q) = 0$ for every $n \geq 0$, $r \geq 1$, and odd $k < 10$.
\end{thm}

\section{Picard group} \label{sec:pic}

We use the exponential sequence to parlay our knowledge about the singular cohomology of $\qpnd$ into knowledge of $\textrm{Pic}(\qpnd)_{\Q}$. Taking homology of the exponential sequence for $\qpnd$ and tensoring by $\Q$ gives:

$$H^1(\qpnd, \OO) \otimes \Q \rightarrow \textrm{Pic}(\qpnd) \otimes \Q \rightarrow H^2(\qpnd, \Z) \otimes \Q.$$

By Proposition \ref{oddratcoh}, $h^1(\qpnd,\Q) = 0$ so $$H^1(\qpnd, \OO) = H^{0,1}(\qpnd) \subset H^1(\qpnd, \C) = 0.$$

Hence $\textrm{Pic}(\qpnd) \otimes \C$ injects into $H^2(\qpnd, \C)$, and its image is $H^{1,1}(\qpnd)$.  Since $\textrm{Pic}(\qpnd) \otimes \C$ is nontrivial and $$H^{2,0}(\qpnd) \oplus H^{1,1}(\qpnd) \oplus H^{0,2}(\qpnd) = H^2(\qpnd, \C) \simeq \C^2,$$ we conclude that $H^{2,0}(\qpnd) = H^{0,2}(\qpnd) = 0$ and $$\textrm{Pic}(\qpnd) \otimes \Q \simeq H^2(\qpnd, \Q).$$  In particular, we have:

\begin{thm} \label{PicardRank}
The Picard rank of $\qpnd$ is equal to $h^2(\qpnd,\Q) = 2$.
\end{thm}

\subsection{Construction of Divisors} \label{sec:ConstructionOfDivisors}
We begin by defining three divisors on $\qg1d$.  We define the quotient map

$$\varphi: \Mdepsilon2 \rightarrow \qg1d$$

by

$$\varphi (C, p_1, ... ,p_d) = (C, p_1 + ... + p_d).$$

\begin{defn}
A point $(C,D) \in \qg1d$, where $D = p_1 + ... + p_d$ is in $D_b$ if any $p_i = p_j$, i.e. if $D$ contains any point of multiplicity $\geq 2$.
\end{defn}

\begin{defn}
Fix a complex number $\rho \in \C$.  A point $(C,D) \in \qg1d$ is in $D_j(\rho)$ if the $j$-invariant of $C$ is $\rho$.
\end{defn}

For any $\rho$ and $\rho'$, $D_j(\rho)$ is linearly equivalent to $D_j(\rho')$.  We will write simply $D_j$ for the linear equivalence class.

For constructing the next divisor, we begin by defining a divisor $B_{\Penab}$ on $\overline{M}_{1,0|2}$.

\begin{defn}
Let $\Pen$ be a pencil of plane cubics through 8 general points in $\p^2$.  Let $a$ and $b$ be two base points of this pencil.  By the universal property of $\overline{M}_{1,0|2}$, this gives a map $\p^1 \rightarrow \overline{M}_{1,0|2}$.  Let the curve $B_{\Penab}$ be the image of this map.
\end{defn}

Since $\overline{M}_{1,0|2}$ is 2-dimensional, $B_{\Penab}$ is a divisor.  To define the divisor $D_{fd}$, we will use the following maps:

$$ \label{UsefulMdepsilon2}
\begin{CD}
\Mdepsilon2 @> \varphi >> \qg1d \\
@V \pi VV  \\
\overline{M}_{1,0|2}\\
\end{CD}
$$

where $\pi: \Mdepsilon2 \rightarrow \overline{M}_{1,0|2}$ is the projection given by $\pi: (C,{q_1,...,q_d}) \mapsto (\tilde{C}, {q_1,q_2})$ where $\tilde{C}$ is obtained from $C$ by contracting any components that become unstable after forgetting the points $q_3,...,q_d$.

\begin{defn}
$$D_{fd}(\Penab) := \varphi_*(\pi^*(B_{\Penab})).$$
\end{defn}

$B_{\Penab}$ is a divisor in $\overline{M}_{1,0|2}$, so because $\varphi$ is finite $D_{fd}$ is a divisor in $\qg1d$.  Concretely, a point $(C,D) \in \qg1d$ is in $D_{fd}(\Penab)$ if there is an isomorphism $\rho: \tilde{C} \rightarrow \Pen_t$ from a component $\tilde{C}$ of $C$ to a fiber of the pencil such that $a$ and $b$ are both in the image $\rho(D|_{\tilde{C}})$.  $D_{fd}(\Penab)$ is so named because the pencil $\Pen$ with the choice of $a$ and $b$ determines a set of fixed distances on each genus 1 curve.

For any two $\Penab$ and $\Pen'_{a',b'}$ $B_{\Penab}$ is linearly equivalent to $B_{\Pen'_{a',b'}}$ in $\overline{M}_{1,0|2}$ because one can explicitly construct a family over $\p^1$ of pencils of plane cubics with two $-1$ sections, such that $\Penab$ is one fiber of the family and $\Pen'_{a',b'}$ another.  $D_{fd}(\Penab)$ and $D_{fd}(\Pen'_{a',b'})$ are then also linearly equivalent and we will write simply $D_{fd}$ for the linear equivalence class.

To define divisors on $\qpnd$ when $n>1$, start with the open locus $Y \subset \qpnd$ of stable quotients which correspond to a map $f: C \rightarrow \pnmo$ where $C$ is a smooth genus 1 curve and the map is nondegenerate in the sense that the image $f(C)$ does not lie in any hyperplane.  Now fix a hyperplane $H \subset \pnmo$.  There is a map $$\Phi: Y \rightarrow \qg1d$$ sending $(C, f)$ to $(C, f(C) \cap H)$.

We define divisors on $\qpnd$ by taking the preimage of divisors in $\qg1d$ under $\Phi$ and then taking the closure.  In our definition of the divisors below, we explicitly describe these closures. Note the similarity between this map from $H_2(\qg1d) \rightarrow H_2(\qpnd)$ and the isomorphism described in \cite{CarrellGoresky} for a scheme $X$ with a \textit{good} $\Cstar$ action between the homology groups of $X$ and the shifted sums of the homology groups of the fixed loci.

Where confusion is possible, we will put a superscript 1 to denote divisors in $\qg1d$ and a superscript $n$ to denote divisors in $\qpnd$, but when it is clear what space we are working in, we will drop the superscripts.  Also, for a section $s$ of a line bundle, we will use $|s|$ to denote its divisor of zeros.

\begin{defn}
Fix a hyperplane $H \subset \pnmo$, defined by $a_0 x_0 + ... + a_{n-1} x_{n-1} = 0$.  Let $(C, [s_0,...,s_{n-1}])$ be a point in $\qpnd$, and let $s_H = a_0 s_0 + ... + a_{n-1} s_{n-1}$.   Let $\Dbn(H)$ be the subvariety of points $(C, [s_0,...,s_{n-1}])$ such that either $s_H = 0$ or on some component $\tilde{C}$ the divisor $|s_H|$ contains a point with multiplicity $\geq 2$.
\end{defn}

For any two hyperplanes $H$, $H'$ in $\pnmo$, $\Dbn(H)$ and $\Dbn(H')$ are linearly equivalent.  We will write $\Dbn$ for the linear equivalence class.

\begin{defn}
Fix a complex number $\rho \in \C$.  A point $(C,[s_0,...,s_{n-1}]) \in \qpnd$ is in $\Djn(\rho)$ if the $j$-invariant of $C$ is $\rho$.
\end{defn}

As above, for any $\rho$ and $\rho'$, $\Djn(\rho)$ and $\Djn(\rho')$ are linearly equivalent, and we will write $\Djn$ for the linear equivalence class.

\begin{defn}
Fix a hyperplane $H$.  Let $\Penab$ be as in the definition of $\Dfd1$.  A point $(C, [s_0,...,s_{n-1}]) \in \qpnd$ is in $\Dfdn(H, \Penab)$ if either $s_H = 0$ or $(\widetilde{C},|s_H|) \in \Dfd1$.  ($(C,|s_H|)$ may not be stable, but after contracting unstable rational components to obtain $\widetilde{C}$, $(\widetilde{C},|s_H|)$ will be stable, in $\qg1d$.)
\end{defn}

As above, for any choices of $H, H', \Penab, \Penabprime$, $\Dfdn(H, \Penab)$ is linearly equivalent to $\Dfdn(H', \Penabprime)$.   We will write simply $D_{fd}$ for the linear equivalence class.

We will see that any two of these three divisors are not numerically equivalent, hence not linearly equivalent.  So the line bundles associated to any pair of these divisors generates $\textrm{Pic}(\qpnd) \otimes \Q = H^2(\qpnd, \Q)$.

In Section \ref{sec:FixedLoci}, we introduced the inclusion $i_{0}: \qg1d \rightarrow \qpnd$.  The induced map $i_{0}^*: H^2(\qpnd, \Q) \rightarrow H^2(\qg1d, \Q)$ sends $\Djn$ to $\Dj1$ and $\Dbn$ to $\Db1$, hence is an isomorphism of vector spaces.  Dually, the induced map $i_{0 *} : H_2(\qg1d,\Q) \rightarrow H_2(\qpnd, \Q)$ is an isomorphism.

\subsection{Construction of Curves} \label{sec:Curves}
We begin by defining curves in $\qg1d$.

\begin{defn}
Fix a smooth genus 1 curve $C$ and an effective divisor $\tilde{D}$ on $C$ of degree $d-1$.  Let $(C, D) \in \qg1d$ be in $\gamma_j(C,\tilde{D})$ if $D = \tilde{D} + p$ for some point $p \in C$.
\end{defn}
In other words, $\gamma_j(C,\tilde{D})$ is obtained by fixing $d-1$ points on $C$ and letting the $d^{th}$ point move.  In particular, $\gamma_j(C,\tilde{D}) \simeq C$.  As before, all choices of $C, \tilde{D}$ give a curve in the same linear equivalence class, and we will write $\gamma_j$ for the linear equivalence class.

\begin{defn}
Fix a pencil of plane cubics $\Pen$ and a base point $a$ of the pencil.  Let $\gamma_t(\Pen, a)$ be the family $(\Pen_t, d \cdot a)$.
\end{defn}
As before, all choices give a curve in the same linear equivalence class, and we will write $\gamma_t$ for the linear equivalence class.

$\gamma_j$ and $\gamma_t$ generate $H_2(\qg1d, \Q)$.  Their images under the induced map $i_{0 *} : H_2(\qg1d,\Q) \rightarrow H_2(\qpnd, \Q)$ are curves generating $H_2(\qpnd, \Q)$.  Let $\gamma_j^n = i_{0 *} \gamma_j^1$ and $\gamma_t^n = i_{0 *} \gamma_t^1$.  When there is no source of confusion, we drop the subscripts.  As the inclusion $i_0$ induces isomorphisms on both the homology and cohomology of $\qg1d$ with $\qpnd$, all intersection numbers computed in $\qg1d$ are also valid in $\qpnd$, so for simplicity we will always compute them in $\qg1d$.

The intersection multiplicities between the divisors and curves defined on $\qpnd$ are given in the following table.

\begin{equation} \label{Intersections}
  \begin{tabular}{ | l | c | c | r | }
    \hline
    &  $\gamma_j$ & $\gamma_{t}$ \\ \hline
    $D_b$    & $d-1$     & $-\frac{d(d-1)}{2}$ \\ \hline
    $D_j$    & 0       & 12  \\ \hline
    $D_{fd}$ & $24(d-1)$ & 0\\
    \hline
  \end{tabular}
\end{equation}

\begin{proof}
The intersections are straightforward to compute, with the exception of $\gamma_t \cdot D_b$, which we do here.

Recall that $\qg1d$ is the quotient of $\Mdepsilon2$ by $S_d$, let $\varphi: \Mdepsilon2 \rightarrow \qg1d$ denote the quotient map.  The preimage of $D_b$ under this map is $\coprod_{i < j} \Delta_{i j}$, where $\Delta_{i j}$ is the divisor where $p_i = p_j$ in $\Mdepsilon2$.

We defined a pencil $\Pen \rightarrow B = \p^1$ with $d$ sections, namely $d$ times the section $\sigma$ given by a base point $a$.  By the universal property of $\qg1d$, this defines a map $\mu: B \rightarrow \qg1d$, and we defined the curve $\gamma_t$ as the image $\mu(B)$.  By definition

$$\gamma_t \cdot D_b = c_1(\mu^* \N_{D_b/\qg1d}).$$

By the universal property of $\Mdepsilon2$, $\mu$ in fact factors through $\Mdepsilon2$, and we have $\mu': B \rightarrow \Mdepsilon2$.  Let us write $C_t$ for the image $\mu'(B)$ in $\Mdepsilon2$.  %By push pull (??), we have

So
$$c_1(\mu^* \N_{D_b/\qg1d}) = c_1(\mu'^* \varphi^* \N_{D_b/\qg1d}) = c_1(\mu'^* \coprod_{i<j} \N_{\Delta_{i j}/\Mdepsilon2} ) = \frac{d(d-1)}{2} c_1(\mu'^* \N_{\Delta_{1 2}/\Mdepsilon2})).$$

It remains to compute the degree of $\N_{\Delta{1 2}}$ on $C_t$.  In the remainder, we denote $\Delta_{1 2}$ simply by $\Delta$ and $\Mdepsilon2$ simply by $M$.

First, we reduce the problem to one involving only the vertical tangent space for the map $\Mdepsilon2 \rightarrow \Moobar$, that is, we avoid deformations that change the modulus of the underlying curve.

$$
\begin{CD}
@. 0 @. 0 @. \\
@. @VVV @VVV @. \\
0 @>>> verT_{\Delta} @>>> verT_{M} @ >>> \NN_{\Delta / M} @ >>> 0\\
@. @VVV @VVV @. \\
0 @>>> T_{\Delta} @>>> T_{M} @ >>> \N_{\Delta / M} @ >>> 0\\
@. @VVV @VVV @. \\
0 @>>> T_{\Moobar} @>>> T_{\Moobar} @ >>> 0 @ >>> 0\\
@. @VVV @VVV @. \\
@. 0 @. 0 @. \\
\end{CD}
$$

By the snake lemma, $\N_{\Delta/M} = \NN_{\Delta/M}$.  To compute $\NN_{\Delta / M}$, we use the fact that the sum of the tangent spaces at each marked point surjects onto $verT_{\Delta}$.

$$
\begin{CD}
@. 0 @. 0 @. \\
@. @VVV @VVV @. \\
0 @>>> K_1 @>\simeq >> K_2 @>>> 0 @>>> 0\\
@. @VVV @VVV @. \\
0 @>>> \oplus_{i=2}^d T_{p_i} @>\Delta \oplus id^{d-2} >> \oplus_{i=2}^d T_{p_i} @ >>> Q @ >>> 0\\
@. @VVV @VVV @. \\
0 @>>> verT_{\Delta} @>>> verT_{M} @ >>> \NN_{\Delta / M} @ >>> 0\\
@. @VVV @VVV @. \\
@. 0 @. 0 @. \\
\end{CD}
$$

Again by the snake lemma, $\NN_{\Delta /M}$ is isomorphic to the quotient $Q$.  Finally we can compute $$\N_{\Delta/M} |_{C_t} = Q |_{C_t} = \bigoplus_{i = 1}^{d} T_{p_i} |_{C_t} - \bigoplus_{i = 2}^{d} T_{p_i} |_{C_t} = T_{p_1} |_{C_t} = -\psi_1 |_{C_t} = -1,$$
where $\psi_1$ is the standard cotangent class to $p_1$.

We conclude $c_1(\mu'^* \N_{\Delta_{1 2}/\Mdepsilon2}))=-1$ and $$c_1(\mu^* \N_{D_b/\qg1d}) = c_1(\mu'^* \coprod_{i<j} \N_{\Delta_{i j}/\Mdepsilon2} ) = -\frac{d(d-1)}{2},$$ as desired.
\end{proof}

We deduce from Table \ref{Intersections} that

\begin{lem} \label{divisor relation}
$D_{fd} = d(d-1)D_j + 24 D_b.$
\end{lem}

\subsection{Nef Cone}
Recall that on a variety $X$, a divisor $D$ is called nef if $D \cdot C = 0$ for every irreducible curve $C \subset X$.  To show $D$ is nef, it is sufficient to prove that it is base point free, because given any curve $C$ and a point $x \in C$, if $D$ is base point free, there is $D'$ numerically equivalent to $D$ such that $D'$ does not contain $x$, so it does not contain $C$, so $D \cdot C = D' \cdot C \geq 0$.

\begin{thm} \label{NefCone}
The nef cone of $\qpnd$ is bounded by the divisors $D_j$ and $D_{fd}$.
\end{thm}

\begin{proof}
We begin by showing $D_j$ is base point free, hence nef.  Take any point $x = (C, [s_0,...,s_{d-1}]) \in \qpnd$.  Fix $\rho$ not equal to the $j$-invariant of $C$.  Then $D_j(\rho)$ is a representative of the class $D_j$ which does not contain $x$.

As computed in Table \ref{Intersections}, $D_j \cdot \gamma_j = 0$, and $\gamma_j$ is an effective curve.  As $D_j$ is nef and has zero intersection with an effective curve, it is on the boundary of the nef cone.

Now we show that $D_{fd}$ is also base point free.  Again take any point $x = (C, [s_0,...,s_{d-1}]) \in \qpnd$.  At least one $s_i \neq 0$.  Let $H$ be the hyperplane $x_i = 0$ in $\pnmo$.  Let $y := (C, |s_H|) \in \qg1d$.  Using again the maps from diagram \ref{UsefulMdepsilon2}, $\pi(\varphi^{-1}(y))$ is a finite set in $\overline{M}_{1,0|2}$.  By a dimension count, we can choose a pencil of plane cubics $\Pen$ and two basepoints $a,b$ of $\Pen$ such that the associated curve $B_{\Penab} \subset \overline{M}_{1,0|2}$ does not meet $\pi(\varphi^{-1}(y))$.  For such a choice of $\Penab$, $D_{fd}(H,\Penab)$ is a representative of the class $D_{fd}$ which does not contain the point $x$.  Hence $D_{fd}$ is base point free, hence nef.

Again, it was computed in Table \ref{Intersections}, that $D_{fd} \cdot \gamma_{t} = 0$.  As $D_{fd}$ is nef and has zero intersection with an effective curve, it is on the boundary of the nef cone.
\end{proof}

\subsection{Projectivity}
The above analysis of the cone of nef divisors and its dual, the cone of effective curves, yields a simple proof of the projectivity of $\qpnd$.

\begin{cor} \label{Projective}
$\qpnd$ is a projective scheme.
\end{cor}

\begin{proof}
If $d = 1$, $\overline{Q}_1 (\pnmo,1) \simeq \pnmo \times \Moobar$ \cite{MOP} Section 2.5.  Hence for $d=1$ $\qpnd$ is projective.

Now assume $d>1$.  Let $D = D_j + D_{fd}$.  We show $D$ is ample, by checking positivity on the closure of the cone of effective curves.  As $\gamma_j$ and $\gamma_t$ span the cone of effective curves, any curve class in the closure of the cone of effective curves can be represented as $E_{a,b} = a \gamma_j + b \gamma_t$ for $a,b$ each nonnegative and not both zero.  By the computation in Table \ref{Intersections}, $D \cdot E_{a,b} = 24(d-1)a + 12b$, which is strictly positive by our assumptions on $(a,b)$ and the assumption $d>1$.  We conclude $D$ is ample, and hence $\qpnd$ is projective for all $d$.
\end{proof}

\subsection{Effective Cone}
Boucksom, Demailly, Paun and Peternell showed that for a projective variety (more generally, compact K\"{a}hler manifold) the cone of pseudo-effective divisors is dual to the cone of moving curves \cite{BDPP}.  For a variety $X$, a moving curve is a linear equivalence class of curves $[\gamma]$ such that for any point $x \in X$, there is a representative $\gamma$ of $[\gamma]$ such that $x \in \gamma$.

\begin{thm} \label{EffCone}
The effective cone of $\qpnd$ is bounded by the divisors $D_j$ and $D_b$.
\end{thm}

\begin{proof}
By construction, the divisors $D_j$ and $D_b$ are effective.  To show that in fact they lie on the boundary, we demonstrate for each a moving curve that has intersection number 0 with the divisor.  Note that to show a curve class $[\gamma]$ is moving, it suffices to show that through any \textit{general} point $y \in X$ there is a representative $\gamma$ such that $y \in \gamma$.  This is because for any $x \in X$, we can take a 1-parameter family of general points $y_t \rightarrow x$ as $t \rightarrow 0$.  If for all $t \neq 0$ there is a curve $\gamma_t$ in the class $[\gamma]$ through $y_t$, then the limit of these curves $\gamma_0$ contains $x = y_0$ and is also in the class $[\gamma]$.

For $D_j$, the curve $\gamma_j$ serves this purpose.  In Section \ref{sec:Curves}, we saw that $\gamma_j \cdot D_j = 0$.  It remains to see that $\gamma_j$ is a moving curve.  Take a general point $(C, D) \in \qpnd$, and by generality assume $C$ is smooth.  Throw away a point of $D$ and let $\tilde{D}$ be the resulting divisor of degree $d-1$.  The representative $\gamma_j(C,\tilde{D})$ contains $(C,D)$, so $\gamma_j$ is a moving curve, as desired.

For the divisor $D_b$, we demonstrate the existence of a moving curve which does not intersect it.  In the Hilbert scheme of degree $m+1$ curves $\xi$ in $\p^m$, let $X$ be the open subscheme where the curve is smooth.  $X$ has dimension $(m+1)^2$.  For any point $z \in \p^m$, let $X_p$ be the subscheme of $X$ of curves $\xi$ passing through $z$.  $X_p$ has codimension $m-1$ in $X$.  Take $d$ general points $z_1,...,z_d$, and let $X_{z_1,...,z_d} = X_{z_1} \bigcap ... \bigcap X_{z_d}$.  If we take $m \geq d-3$, $X_{z_1,...,z_d}$ will have dimension $\geq 4$, in particular, it will be nonempty.

Take a curve $\gamma \subset X_{z_1,...,z_d}$.  Restricting the universal family of the Hilbert scheme to $\gamma$ gives is a 1 dimensional family of smooth genus 1 curves in $\p^m$, each with $d$ \textit{distinct} marked points $z_1,...,z_d$.  By the universal property of $\qg1d$, this maps to a curve in $\qg1d$, and let $\gamma_{nt} \subset \qg1d$ be its closure.

For all $(C,D) \in \gamma_{nt}$, $D$ is an effective divisor of $d$ distinct points, so $\gamma_{nt} \cdot D_b = 0$.  It remains to see that  $\gamma_{nt}$ is a moving curve.  Fix any general point $(C,D) \in \qg1d$.  By genericity, $C$ is smooth and $D$ consists of distinct points.  Pick any embedding of $i: C \hookrightarrow \p^m$, and label the points of $i(D)$ as $z_1,...,z_d$.  Now the image $\gamma_{nt}$ of any $\gamma \subset X_{z_1,...,z_d}$ contains $(C,D)$, as desired.  We conclude that $D_b$ is a boundary of the effective cone of $\qpnd$.
\end{proof}

\subsection{Canonical Divisor}

\begin{thm} \label{Canonical}
$K_{\qpnd}  =  \frac{\left( d-11 + (d-1)(n-1) \right)}{12} D_j - n D_b.$
\end{thm}

\begin{proof}
We begin with the $n=1$ case.

The map $\varphi: \Mdepsilon2 \stackrel{/S_d}{\rightarrow} \qg1d$  is ramified exactly along the divisor $\sum_{i<j} \Delta_{i j}$.  We can use the Riemann-Hurwitz formula to deduce $K_{\qg1d}$ from the canonical divisor of $\Mdepsilon2$, $$K_{\Mdepsilon2} = 13\lambda - 2 \Delta_{nodal} + \psi_{tot},$$ which was computed by Hassett \cite{Hassett}.  Here $\Delta_{nodal}$ is the locus where the genus 1 curve is nodal, $\psi_{tot} = \sum_i \psi_i$ and $\lambda$, $\psi_i$ are the standard Hodge and cotangent classes.

By Riemann-Hurwitz,
\begin{equation}
\varphi^* K_{\qg1d} = K_{\Mdepsilon2} - R\\
= 13 \lambda - 2 \Delta_{nodal} + \psi_{tot} - \sum_{i<j} \Delta_{i j}
\end{equation}

Writing $K_{\qg1d} = \alpha D_j + \beta D_b$, $\varphi^* K_{\qg1d} = 12 \alpha \lambda + \beta \sum_{i<j} \Delta_{i j}$ and

$$12 \alpha \lambda + \beta \sum_{i<j} \Delta_{i j} = (13 - 24 + d) \lambda - \sum_{i<j} \Delta_{i j}$$

so

\begin{equation} \label{Kqg1d}
K_{\qg1d} = \frac{d-11}{12}D_j - D_b.
\end{equation}

For $n>1$, we use the inclusion $i_0: \qg1d \hookrightarrow \qpnd$, which gives the short exact sequence

\begin{equation} \label{Kqg1d and Kqpnd ses}
0 \rightarrow T_{\qg1d} \rightarrow i_0^* T_{\qpnd} \rightarrow \N_{\qg1d/\qpnd} \rightarrow 0.
\end{equation}

Take a family
$
\begin{CD}
@. \CC  @.  \\
\pi @VVV @AAA @. \sigma_1 + ... + \sigma_d \\
@. B @.\\
\end{CD}
$

over a curve $B \subset i_0(\qg1d)$.  Then the normal bundle is $$\N_{\qg1d/\qpnd}|_B = \pi_*\left(\OO(\sum \sigma_i)^{n-1}\right).$$  We can use Grothendieck-Riemann-Roch to compute the first chern class of this bundle.

\begin{eqnarray*}
& & ch\left(\pi_* \OO(\sum \sigma_i)^{n-1}\right)  =  \pi_*\left( ch \left( \OO(\sum \sigma_i)^{n-1}\right) td(T_{\CC/B}) \right)\\
& = & \pi_*\left[ \left( (n-1) + (n-1) (\sum \sigma_i) + \frac{(n-1)^2 - 2\left(\stackrel{n-1}{2}\right)}{2} (\sum \sigma_i)^2 \right) \left( 1 - \frac{1}{2} \gamma + \frac{1}{12} (\gamma^2 + \eta) + ... \right) \right]\\
& = & \pi_* \left[ (n-1)\left(1 + (\sum \sigma_i) + \frac{1}{2} (\sum \sigma_i)^2 \right) \left(1-\frac{1}{2} \gamma + \frac{1}{12} (\gamma^2 + \eta) + ... \right)  \right]
\end{eqnarray*}

Taking the degree 1 part of each side,

\begin{eqnarray*}
c_1 \left(\pi_* \OO(\sum \sigma_i)^{n-1}\right) = \pi_* \left[ (n-1) \left( \frac{1}{12} (\gamma^2 + \eta) - \frac{1}{2} \gamma \cdot (\sum \sigma_i) + \frac{1}{2} (\sum \sigma_i)^2 \right) \right] \\
= (n-1) \left( \lambda - \frac{\pi_*(\gamma \cdot \sum \sigma_i)}{2} + \frac{\pi_*(\sum \sigma_i)^2}{2} + \frac{\pi_*(\sum_{i \neq j} \sigma_i \cdot \sigma_j}{2} \right)
\end{eqnarray*}

$\OO(\sigma_i^2)$ is the normal bundle to $\sigma_i$, so $$\pi_*(\sigma_i^2) = - \psi_i.$$  Meanwhile $$\gamma \cdot \sigma_i = c_1(\omega_{\CC/B} |_{\sigma_i}) = \psi_i.$$

So
\begin{eqnarray*}
c_1 \left( \pi^* \N_{\qg1d/\qpnd} \right)
%= (n-1) \left( \lambda - \frac{\psi_{tot}}{2} + \frac{-\psi_{tot}}{2} + \frac{2 \sum_{i<j} \Delta_{i j}}{2} \right)
= (n-1) \left( \lambda - \psi_{tot} + \sum_{i<j} \Delta_{i j} \right).
\end{eqnarray*}

Writing $$c_1 \left( \N_{\qg1d/\qpnd} \right) = \alpha D_j + \beta D_b,$$

$$12 \alpha \lambda + \beta \sum_{i<j} \Delta_{i j} = (n-1) ((1-d)\lambda + \sum_{i<j} \Delta_{i j})$$

and we conclude

$$c_1 \left( N_{\qg1d/\qpnd} \right) = \frac{(n-1)(1-d)}{12} D_j + (n-1) D_b.$$

Taking the first chern class of each term of (\ref{Kqg1d and Kqpnd ses}) and using the expression \ref{Kqg1d} for $K_{\qg1d}$, we conclude

\begin{eqnarray*}
c_1(i_0^* T_{\qg1d}) & = & c_1(T_{\qg1d}) + c_1(\N_{\qg1d/\qpnd}) \\
& = & - \left( (d-11) \frac{D_j}{12} - D_b \right) + (n-1) \left( (1-d) \frac{D_j}{12} + D_b \right)
\end{eqnarray*}

Hence

\begin{equation*}
i_0^* K_{\qg1d})  =  \frac{\left( d-11 + (d-1)(n-1) \right)}{12} D_j - n D_b.
\end{equation*}

Because $i_0^*$ is an isomorphism of the Picard groups of $\qg1d$ and $\qpnd$, we conclude that for all $n \geq 1$,
$$K_{\qpnd})  =  \frac{\left( d-11 + (d-1)(n-1) \right)}{12} D_j - n D_b,$$
as desired.
\end{proof}

\begin{prop}
$\qpnd$ is Fano iff $n(d+2)(d-1) < 20$.
\end{prop}

\begin{proof}
Recall that a scheme $X$ is Fano if $-K_X$ ample.  We use Lemma \ref{divisor relation} to write $-K_{\qpnd}$ in terms of the divisors bounding the nef cone.  We find

$$-K_{\qpnd} = \frac{1}{24} \left( n D_{fd} - ( n(d+2)(d-1) - 20 ) D_j \right).$$
\end{proof}

\section{Rational connectedness} \label{sec:ratconn}
Recall that a scheme $X$ is \textit{rationally connected} if for any two general points $x_1, x_2 \in X$, there is a chain of rational curves $R_1, ..., R_k$ such that $x_1 \in R_1$ and $x_2 \in R_k$.

\begin{thm} \label{thm:ratconn}
$\qpnd$ is rationally connected.
\end{thm}

\begin{proof}
First, consider the case of $\qg1d$.  As we need show only that two general points are connected by a chain of rational curves, it suffices to fix two general points $(C,D), (C',D')$ in $\widetilde{U}_d$.  By Proposition \ref{Ud is fiber bundle}, $\widetilde{U}_d \rightarrow \tilMoo$ is a fiber bundle with fibers $\mathbb{P}^{d-1}/\hat{H}$.  So the fibers are rationally connected.  (To construct a line in $\pdmo/\hat{H}$ between any two points $p_1$ and $p_2$, take a cover $\pi: \pdmo \rightarrow \pdmo/\hat{H}$, and two preimages $q_1$, $q_2$ of $p_1$, $p_2$.  There is a line $L \subset \pdmo$ between $q_1$ and $q_2$.  The quotient $\pi(L)$ is rational because any quotient of $\p^1$ by a finite group is isomorphic to $\p^1$, and so $\pi(L)$ is a rational curve in $\pdmo/\hat{H}$ connecting $p_1$ and $p_2$.)

Let $R_1$ be a rational curve between $(C,D)$ and $(C, dp)$, where $dp$ is the divisor $d$ times a point $p$.  Let $R_2$ be a rational curve between $(C, dp)$ and $(C', dp')$ (a copy of $\Moobar$ in $\qg1d$).  And let $R_3$ be a rational curve between $(C',dp')$ and $(C', D')$.  This chain of rational curves connects the two general points we started with.

In the general case $n>1$, again fix two general stable quotients $(C, [s_0,...,s_{n-1}])$ and $(C', [s'_0,...,s'_{n-1}])$ whose underlying curve is a smooth genus 1 curve other than $\xi_4$ or $\xi_6$.  Let $R_1$ be the closure of $(C, [s_0, \epsilon s_1, ..., \epsilon s_{n-1}])$, $\epsilon \in \C$, and $R_5$ the closure of $(C', [s'_0, \epsilon s_1, ..., \epsilon s_{n-1}])$, $\epsilon \in \C$.  By genericity, $s_0$ and $s'_0$ are both nonzero.

Recall the inclusion $i_0:\qg1d \rightarrow \qpnd$ (Equation \ref{qg1d in qgnd}).  $R_1$ and $R_5$ both intersect the image $i_0(\qg1d)$, namely at $(C, |s_0|)$ and $(C', |s'_0|)$, where $|s|$ denotes the zero set of the section.  As discussed above, there is a chain of rational curves $R_2, R_3, R_4$ in $i_0(\qg1d)$ connecting those two points.  Taken together, $R_1, ... ,R_5$ is then a chain of rational curves connecting the original two general points, as desired.
\end{proof}

\vspace{+10pt}
\noindent Yaim Cooper \\
Department of Mathematics \\
Princeton University \\
Princeton NJ, 08540 \\
{\tt yaim@math.princeton.edu} \\


\begin{thebibliography}{}

\bibitem [B74] {BB1} {\sc Bia{\l}ynicki-Birula, A.}, {\em On Fixed Points of Torus Actions on Projective Varieties}, {Bulletin De L'Acad\^{e}mie Polonaise des Sciences, S\'{e}rie des sciences math. astr. et phys.}, {\bf 22}({11}):{1097--1101} (1974).

\bibitem [BDPP04] {BDPP} {\sc Boucksom, S. and Demailly, J.P. and Paun, M. and Peternell, T.}, {\em The pseudo-effective cone of a compact K\"{a}hler manifold and varieties of negative Kodaira dimension}, {arXiv:math/0405285v1}, (2004).

\bibitem [CCC09] {CCC} {\sc Chen, D., Coskun, I., and Crissman, C.}, {\em Towards Mori's program for the moduli space of stable maps}, {arXiv:0905.2947}, (2009).

\bibitem[CHS08] {CHS1} {\sc Coskun, I., Harris, J., and Starr, J.}, {\em The effective cone of the Kontsevich moduli space}, {Canad. Math. Bull.}, {\bf 51}({4}):{519--534} (2008).

\bibitem[CHS09] {CHS2} {\sc Coskun, I., Harris, J., and Starr, J.}, {\em The ample cone of the Kontsevich moduli space}, {Canad. J. Math.}, {\bf 61}({1}):{109--123} (2009).

\bibitem[CG83] {CarrellGoresky} {\sc Carrell, J.B. and R.M. Goresky}, {\em A Decomposition Theorem for the Integral Homology of a Variety}, {Invent. math.}, {\bf 73}:{367--381} (1983).

\bibitem[D71] {Deligne} {\sc Deligne, Pierre}, {\em Th\'{e}orie de Hodge II}, {Publ. Math. IHES}, {\bf 40}:{5-57} ({1971}).

\bibitem[F06] {Fontanari} {\sc Fontanari, C.}, {\em Towards the Cohomology of Moduli Spaces of Higher Genus Stable Maps}, {arXiv:math/0611754v1}, (2006).

\bibitem[F93] {Fulton} {\sc Fulton, William}, {\em Introduction to Toric Varieties}, {Annals of Mathematics Studies}, {131} {Princeton University Press}, {Princeton}, (1993).

\bibitem [GP05] {GetzlerPandharipande}  {\sc Getzler, E. and Pandharipande, R.}, {\em The Betti numbers of $\overline{M}_{0,n}(r,d)$}, arXiv:0502525 (2005).

\bibitem [Ha03] {Hassett} {\sc Hassett, B.}, {\em Moduli spaces of weighted pointed stable curves}, {Adv. Math.}, {\bf 173}:{316--352}, (2003).

\bibitem[HL08] {HuLi} {\sc Hu, Y. and Li, J.}, {\em Genus-One Stable Maps, Local Equations, and Vakil-Zinger's desingularization}, {arXiv:0812.4286}, (2008).

\bibitem[Ho03] {redbook} {\sc Hori, Kentaro et. al}, {\em Mirror Symmetry}, {Clay Mathematics Monographs}, {\bf 1}, {American Mathematical Society}, ({2003}).

\bibitem[I72] {Iversen} {\sc Iversen, B.}, {\em A fixed point formula for actions of tori on algebraic varieties}, {Inventiones Math.}, {\bf 16}:{229-236}, ({1972}).

\bibitem[MOP09] {MOP} {\sc Marian, A. and Oprea, D. and Pandharipande, R.}, {\em The Moduli Space of Stable Quotients}, {arXiv:0904.2992}, (2009).

\bibitem[MF82] {MumfordFogarty} {\sc Mumford, D. and Fogarty, J.}, {\em Geometric Invariant Theory}, {Springer-Verlag}, {Berlin and New York}, (1982).

\bibitem[O06] {Oprea} {\sc Oprea, D.}, {\em Tautological classes on the moduli spaces of stable maps to $\p^r$ via torus actions}, {Advances in Mathematics}, {\bf 207}:{661--690}, (2006).

\bibitem[Pa97] {Pand1} {\sc Pandharipande, R.}, {\em The Canonical Class of $\overline{M}_{0,n}(P^r,d)$ And Enumerative Geometry}, {Internat. Math. Res. Notices }, {\bf 4}:{173–-186}, (1997).

\bibitem[Pa99] {Pand2} {\sc Pandharipande, R.}, {\em Intersections of Q-Divisors on Kontsevich's Moduli Space $\bar{M}_{0,n}(P^r,d)$ and Enumerative Geometry}, {Trans. Amer. Math. Soc.}, {\bf 351}({4}):{1481--1505}, (1999).

\bibitem[Pa09a] {Pandkappa1} {\sc Pandharipande, R.}, {\em The kappa ring of the moduli of curves of compact type: I}, {arXiv:0906.2657}, (2009).

\bibitem[Pa09b] {Pandkappa2} {\sc Pandharipande, R.}, {\em The kappa ring of the moduli of curves of compact type: II}, {arXiv:0906.2658}, (2009).

\bibitem[PP11] {PandPix} {\sc Pandharipande, R. and Pixton, A.}, {\em Relations in the tautological ring}, {arXiv:1101.2236}, (2011).

\bibitem[Po37] {Polya} {\sc Polya, G.}, {\em Kombinatorische Anzahlbestimmungen für Gruppen, Graphen und chemische Verbindungen}, {Acta Mathematica}, {\bf 68}(1):{145--254}, (1937).

\bibitem[Sa89] {Saito} {\sc Saito, M.}, {\em Introduction to Mixed Hodge Modules}, {Ast\'{e}risque}, {\bf 179-180}:{145--162}, (1989).

\bibitem[Sk11] {Skowera} {\sc Skowera, J.}, {\em Bialinciki-Birula decomposition of Deligne-Mumford stacks}, In preparation, (2011).
\bibitem[T10] {Toda} {\sc Toda, Y.}, {\em Moduli spaces of stable quotients and the wall-crossing phenomena}, {arXiv:1005.3743}, (2010).

\bibitem[VZ06] {VakilZinger} {\sc Vakil, R. and Zinger, A.}, {\em A Desingularization of the Main Component of the Moduli Space of Genus-One Stable Maps into $P^n$}, {arXiv:math/0603353}, (2006).

\bibitem[V03] {Voisin} {\sc Voisin, Clare}, {\em Hodge Theory and Complex Algebraic Geometry vol. I}, {Cambridge University Press}, {New York}, (2003).

\bibitem[Z04] {Zinger} {\sc Zinger, A.}, {\em A Sharp Compactness Theorem for Genus-One Pseudo-Holomorphic Maps}, {arXiv:math/0406103}, (2004).
\end{thebibliography}
\end{document}